\documentclass[a4paper,reqno,11pt]{amsart}
\usepackage[margin=1.3in]{geometry}
\usepackage[utf8]{inputenc}	
\usepackage{indentfirst} 
\usepackage{amsmath,amsfonts,amssymb,,mathrsfs}
\usepackage{bbm}
\usepackage[colorlinks=true]{hyperref}
\usepackage{enumerate}
\usepackage{graphicx}
\usepackage{xcolor}

\theoremstyle{definition}
\newtheorem{definition}{Definition}[section]
\theoremstyle{remark}
\newtheorem{remark}[definition]{Remark}
\theoremstyle{theorem}
\newtheorem{theorem}[definition]{Theorem}
\theoremstyle{theorem}

\theoremstyle{remark}

\theoremstyle{theorem}
\newtheorem{corollary}[definition]{Corollary}
\theoremstyle{theorem}
\newtheorem{proposition}[definition]{Proposition}
\theoremstyle{theorem}

\numberwithin{equation}{section}
\allowdisplaybreaks

\def\Xint#1{\mathchoice
    {\XXint\displaystyle\textstyle{#1}}%
    {\XXint\textstyle\scriptstyle{#1}}%
    {\XXint\scriptstyle\scriptscriptstyle{#1}}%
    {\XXint\scriptscriptstyle\scriptscriptstyle{#1}}%
    \!\int}
\def\XXint#1#2#3{{\setbox0=\hbox{$#1{#2#3}{\int}$}
      \vcenter{\hbox{$#2#3$}}\kern-.5\wd0}}

\def\mint{\Xint-}

\definecolor{ao}{rgb}{0.0, 0.5, 0.0}

\DeclareMathOperator*{\aplim}{ap-\lim}

\newcommand{\om}{\omega}
\newcommand{\Om}{\Omega}
\newcommand{\R}{\mathbb{R}}
\newcommand{\C}{\mathbb{C}}
\newcommand{\F}{\mathcal{F}}
\newcommand{\HH}{\mathcal{H}}
\newcommand{\di}{\mathrm{d}}
\newcommand{\M}{\mathbb{M}}
\newcommand{\G}{\mathcal{G}}
\newcommand{\E}{\mathcal{E}}
\newcommand{\e}{\mathrm{E}}
\newcommand{\Ll}{\mathcal{L}}

\title{Brittle fracture in linearly elastic plates}
\date{\today}
\author[S. Almi]{Stefano Almi}
\address[Stefano Almi]{Faculty of Mathematics, University of Vienna, 
Oskar-Morgenstern-Platz 1, 1090 Wien, Austria.}
\email{stefano.almi@univie.ac.at}

\author[E. Tasso]{Emanuele Tasso}
\address[Emanuele Tasso]{Technische Universit\"at Dresden, Faculty of Mathematics, 01062 Dresden, Germany}
\email{emanuele.tasso@tu-dresden.de}


\begin{document}

\maketitle

\begin{abstract}
In this work we derive by $\Gamma$-convergence techniques a model for brittle fracture linearly elastic plates. Precisely, we start from a brittle linearly elastic thin film with positive thickness~$\rho$ and study the limit as~$\rho$ tends to~$0$. The analysis is performed with no a priori restrictions on the admissible displacements and on the geometry of the fracture set. The limit model is characterized by a Kirchhoff-Love type of structure.
\end{abstract}

\section{Introduction}\label{s:intro}

This paper is devoted to the rigorous derivation of a brittle fracture model for elastic plates by means of dimension reduction techniques. The target $(n-1)$-dimensional plate is represented by an open bounded subset~$\omega$ of~$\R^{n-1}$ with Lipschitz boundary~$\partial\omega$. As it is typical in dimension reduction problems, the plate is first endowed with a fictitious thickness~$\rho>0$, so that, in an $n$-dimensional setting, the initial reference configuration is given by the set~$\Om_{\rho}:= \omega \times (-\frac{\rho}{2}, \frac{\rho}{2})$. The starting point of our analysis is the by now classical variational model of brittle fracture in linearly elastic bodies~\cite{fra}
\begin{equation}\label{e:intro1}
\F_{\rho}(u) : = \frac{1}{2} \int_{\Om_{\rho}} \C e(u) {\, \cdot\,} e(u) \, \di x + \HH^{n-1}(J_{u}) \,,
\end{equation}
where the displacement~$u\colon \Om_{\rho} \to \R^{3}$ belongs to the space~$GSBD^{2}(\Om_{\rho})$ of generalized special functions of bounded deformation~\cite{dal},~$e(u)$ is the approximate symmetric gradient of~$u$,~$J_{u}$ stands for the jump set of~$u$,~$\HH^{n-1}$ indicates the $(n-1)$-dimensional Hausdorff measure in~$\R^{n}$, and~$\C$ is the linear elasticity tensor, here characterized by its Lam\'e coefficients~$\lambda$ and~$\mu$. We further refer to Sections~\ref{s:preliminaries} and~\ref{s:setting} for the notation and the precise assumptions.

The aim of our work is to study the limit, in terms of~$\Gamma$-convergence, of the functional~\eqref{e:intro1} as the thickness parameter~$\rho$ tends to~$0$. The literature related to dimension reduction problems in Continuum Mechanics is very rich. In a purely elastic regime, we mention~\cite{per, Cia1997} for the derivation of reduced models of linearly elastic plates, and~\cite{mor1, MR4121138,  mor4, MR1916989,  mul1, mor3, mor2} for a number of nonlinear models for plates and shells obtained as limit of $3$-dimensional nonlinear elasticity. Further applications to the theory of elastic plates and shells can be found in~\cite{hor, hor2, neu, vel}, where the interplay between dimension reduction and homogenization is studied. In an elastoplastic setting, in~\cite{dav3, dav2, dav1, mag} the authors obtained models for thin elastoplastic plates, starting from either linearized or finite plasticity, and also proved the convergence of the corresponding quasistatic evolutions, in the spirit of evolutionary $\Gamma$-convergence~\cite{mie1, mie2}. 

In the context of fracture mechanics, the study of the $\Gamma$-limit of free discontinuity functionals of the form~\eqref{e:intro1} has been considered, for instance, in~\cite{alm, bab2, bab1, bra, fre, bab3}. In particular,~\cite{bab2, bra} are concerned with the nonlinearly elastic case, in which the stored elastic energy density obeys a $p$-growth condition of the form $W(F) \geq C(|F|^{p} - 1)$ which is incompatible with linear elasticity. The papers~\cite{alm, bab3} consider the antiplanar case, where the energy is in the form~\eqref{e:intro1} but the displacement~$u$ is supposed to be orthogonal to the middle surface~$\omega$, so that the dimension reduction problem becomes scalar and is described in terms of $GSBV$-functions (see, e.g.,~\cite[Section~4.5]{amb1}). In~\cite{fre} the authors considered the convergence of quasistatic evolutions in the vectorial case, under the assumption that the crack path is known a priori, is transversal to the middle surface~$\omega$, and cuts the whole of~$\Om_{\rho}$. The geometrical restriction on the fracture set was then removed in~\cite{bab1}, where the $\Gamma$-limit of~$\F_{\rho}$ in~\eqref{e:intro1} has been studied under the restriction~$u \in SBD(\Om_{\rho})$, the space of special functions of bounded deformation~\cite{amb2}. In order to 
ensure that sequences equi-bounded in energy are sequentially relatively compact, the authors had however to assume an a priori bound on the $L^{\infty}$-norm of the displacement~$u$, which is in general not guaranteed by the boundedness of functional~$\F_{\rho}$. Moreover,  the $L^\infty$-assumption simplifies the argument used to characterize the limiting thin plate model and to detect all the admissible displacement belonging to the domain of the $\Gamma$-limit.

The aim and main novelty of our work is to study the limit of~$\F_{\rho}$ in a~$GSBD$-setting, removing the unphysical a priori bound on the norm of the displacement. As in~\cite{bab1}, we prove in Theorem~\ref{t:limit} that the $\Gamma$-limit writes
\begin{displaymath}
\frac{1}{2} \int_{\Om_1} \C_{0} e(u)  {\, \cdot\,} e(u) \, \di x+ \HH^{n-1} (J_{u})
\end{displaymath} 
for $u \in GSBD^{2}(\Om_1)$ such that $e_{i, n}(u) = 0$ for $i = 1, \ldots, n$ and $(\nu_{u})_{n} = 0$ on~$J_u$. Here,~$\nu_{u}$ is the approximate unit normal to~$J_{u}$ and~$\C_{0}$ is the reduced elasticity tensor of the Kirchhoff-Love theory of elastic plates~\cite{Cia1997}, defined in terms of the Lam\'e coefficients as
\begin{displaymath}
\C_{0} \mathrm{E} \cdot \mathrm{E} :=  \frac{2 \lambda \mu}{\lambda + 2\mu}( tr(\e))^{2} + 2\mu |\e|^{2}
\end{displaymath}
for $\e$ in the space of square symmetric matrices of order~$n-1$. The most technical part of our result, which in particular influences the construction of a recovery sequence in the proof of Theorem~\ref{t:limit}, is the characterization of the admissible displacement~$u$ in the limit model. Indeed, in Theorem~\ref{t:KL} we show that~$u$ has a Kirchhoff-Love type of structure: the out-of-plane component~$u_{n}$ does not depend on the vertical variable~$x_{n}$, while the in-plane components~$u_{1},\ldots, u_{n-1}$ satisfy
\begin{equation}\label{e:intro2}
u_{\alpha} (x', x_{n}) = \overline{u}_{\alpha} (x') - x_{n} \partial_{\alpha} u_{n}(x') 
\end{equation}
for $x= (x', x_{n}) \in \Om_{1}$ and $\alpha = 1, \ldots, n-1$, where $\overline{u}_{\alpha} (x') := \int_{-1/2}^{1/2} u_{\alpha} (x', x_{n}) \, \di x_{n}$. In contrast to~\cite{bab1}, due to the lack of integrability of $u$, we cannot conclude $u_n \in GSBV(\omega)$ while we can ensure that at a.e.~$x' \in \omega$ $u_n$ is approximate differentiable. Moreover, the proof of~\eqref{e:intro2} is not based on convolution techniques combined with the study of the distributional symmetric gradient~$\e u$ of~$u$ (see~\cite[Proposition~5.2]{bab1}), as~$\e u$ is not a bounded Radon measure for~$u \in GSBD^{2}(\Om_{1})$. On the contrary, we obtain~\eqref{e:intro2} through an approximation result similar to~\cite[Theorem 1]{cha} and~\cite[Theorem~5]{iur}, which therefore allows us to work with functions that are $W^{1, \infty}$ out of the closure of their jump set. The crucial point in such an approximation is that we need to 
\begin{itemize}
\item guarantee that on large part of the domain~$\Om_{1}$ the $n$-th component of the approximating function~$u_{k}$ is still independent of~$x_{n}$;
\vspace{1mm}
\item control the $\HH^{n-1}$-measure of the projection~$\pi_{n}(\overline{J_{u_{k}}})$ of the closure of the jump set of the approximating sequence~$u_{k}$ on~$\om$ by means of~$\HH^{n-1}( \pi_{n} (J_{u}))$. 
\end{itemize}
The two properties above, together with the fact that actually $\HH^{n-1}( \pi_{n} (J_{u}))= 0$, allow us to apply the Fundamental Theorem of Calculus in the direction~$x_{n}$ to the sequence~$u_{k}$, obtain a first version of~\eqref{e:intro2} for~$u_{k}$, and then conclude by passing to the limit in~$k$ and by further exploiting that the jump set~$J_u$ is transversal to the middle surface~$\om$. This argument is made rigorous in Propositions~\ref{app},~\ref{p:KL2}, and~\ref{vj}. In a similar way to~\cite{bab1}, we show that the jump set~$J_{u}$ takes the form
\begin{displaymath}
J_{u} = (J_{\overline{u}} \cup J_{u_n} \cup J_{\nabla u_n} ) \times \bigg( -\frac12 , \frac12 \bigg)\,,
\end{displaymath}
where $\overline{u} : = ( \overline{u}_{1}, \ldots, \overline{u}_{n-1})$, concluding the description of the admissible displacements.

Here we point out that it is possible to make use of the recent results obtained in~\cite{CCS20} in the context of $GSBD^p(\Omega)$ with $p>1$, in order to deduce the Kirchoff-Love structure of any admissible displacement (see also \cite{fri} for the existence of an approximate gradient for functions in $GSBD^p(\Omega)$ with $p >1$). Nevertheless, by using Proposition \ref{app}, it is easy to verify that the same result of Theorem \ref{t:KL} can be proved for $u$ belonging to \begin{align}
\{ u \in GSBD(\Om_{1}):  \, e_{i, n} (u) = 0 \text{ in~$\Om_{1}$,}
\ \mathcal{H}^{n-1}(J_u)< \infty, \ \text{and $(\nu_{u})_{n} = 0$ on~$J_{u}$}\}\,, \nonumber
\end{align}
instead of $\mathcal{KL}(\Om_{1})$ (see \eqref{e:KL}),
with the obvious modifications $\nabla u_n \in GSBD(\omega)$ and $\overline{u} \in GSBD(\omega)$. As a matter of fact, our method highlights the fact that \emph{the nature of the Kirchoff-Love structure does not depend on a major integrability of the approximate symmetric gradient}. In addition, we consider our procedure self contained, in the sense that we do not need to deal with any kind of Korn's or Korn-Poincar\'e's inequalities.

Finally, we extend the $\Gamma$-convergence result of Theorem~\ref{t:limit} to the case of non-homogeneous Dirichlet boundary conditions in Corollary~\ref{c:bc} and further discuss the convergence of minima and minimizers in Theorem~\ref{t:compactness2} and Corollary~\ref{c:conv_minimizer}. 


\section{Preliminaries and notation}\label{s:preliminaries}

We briefly recall here the notation used throughout the paper. For $n , k \in \mathbb{N}$, we denote by~$\mathcal{L}^{n}$ the Lebesgue measure in~$\R^{n}$ and by~$\HH^{k}$ the $k$-dimensional Hausdorff measure in~$\R^{n}$. The symbol~$\mathbb{M}^{n}$ stands for the space of square matrices of order~$n$ with real coefficients, while~$\mathbb{M}^{n}_{s}$ indicates the subspace of~$\mathbb{M}^{n}$ of squared symmetric matrices of order~$n$. For every $r>0$ and every $x \in \R^{n}$, we denote by~$B_{r}(x)$ the open ball in~$\R^{n}$ of radius~$r$ and center~$x$. We will indicate with~$\{e_{1}, \ldots, e_{n}\}$ the canonical basis of~$\R^{n}$ and with~$\mathbbm{1}_{E}$ the characteristic function of a set~$E \subseteq \R^{n}$. For every $\xi \in \mathbb{S}^{n-1}$,~$\pi_{\xi}$ stands for the projection over the subspace~$\xi^{\perp}$ orthogonal to~$\xi$. If~$\xi = e_{i}$ for $i = 1, \ldots, n$, we use the symbol~$\pi_{i}$.

For every $U\subseteq \R^{n}$ open, we denote by~$\mathcal{M}_{b}(U)$ and~$\mathcal{M}^{+}_{b}(U)$ the set of bounded Radon measures and of positive bounded Radon measures in~$U$, respectively. Let $m \in \mathbb{N}$ with $m\geq 1$. For every $\mathcal{L}^{n}$-measurable function $v \colon U \to \R^{m}$ and every~$x \in U$ such that
\begin{displaymath}
\limsup_{r \searrow 0} \, \frac{\Ll^{n}( U \cap B_{r} (x)) }{r^{n}} >0\,,
\end{displaymath}
we say that $a \in \R^{m}$ is the approximate limit of~$v$ at~$x$ if
\begin{displaymath}
\lim_{r \searrow 0} \,  \frac{\Ll^{n}( U \cap B_{r} (x) \cap \{ | v - a| > \epsilon\} ) }{r^{n}} = 0 \qquad \text{for every $\epsilon > 0$}\,.
\end{displaymath}
In this case, we write
\begin{displaymath}
\aplim_{y \to x} v(y) = a\,.
\end{displaymath}
We say that $x \in U$ is an approximate jump point of~$v$, and we write $x \in J_v$, if there exist $a, b \in \R^{m}$ with $a \neq b$ and $\nu \in \mathbb{S}^{n-1}$ such that
\begin{displaymath}
\aplim_{\substack{ y \to x \\ (y - x) \cdot \nu >0}}\, v(x) = a \qquad \text{and} \qquad \aplim_{\substack{ y \to x \\ (y - x) \cdot \nu <0}}\, v(x) = b\,.
\end{displaymath}
In particular, for every $x \in J_v$ the triple $(a, b, \nu)$ is uniquely determined up to a change of sign of~$\nu$ and a permutation of~$a$ and~$b$. We indicate such triple by~$(v^{+}(x), v^{-}(x), \nu_{v}(x))$. The jump of~$v$ at~$x \in J_v$ is defined as $[v](x) := v^{+}(x) - v^{-}(x)$. We denote by~$(\nu_{v})_{i}$ the components of~$\nu_{v}$, for $i = 1, \ldots, n$.

The space $BV(U;\R^{n})$ of functions of bounded variation is the set of $u\in L^{1}(U;\R^{n})$ whose distributional gradient~$Du$ is a bounded Radon measure on~$U$ with values in~$\mathbb{M}^{n}$. Given $u\in BV(U;\R^{n})$, we can write $Du=D^{a}u+D^{s}u$, where~$D^{a}u$ is absolutely continuous and~$D^{s}u$ is singular w.r.t.~$\Ll^{n}$. The set~$J_{u}$ is countably $(\HH^{n-1}, n-1)$-rectifiable and has approximate unit normal vector~$\nu_{u}$, while the density~$\nabla{u} \in L^{1}(U;\mathbb{M}^{n})$ of~$D^{a}u$ w.r.t.~$\Ll^{n}$ coincides a.e.~in~$U$ with the approximate gradient of~$u$, that is, for a.e.~$x \in U$ it holds
\begin{displaymath}
\aplim_{y \to x} \, \frac{u(y) - u(x) - \nabla{u}(x) \cdot (y - x) }{| x - y| } = 0\,.
\end{displaymath}

The space $SBV(U;\R^{n})$ of special functions of bounded variation is defined as the set of all $u\in BV(U;\R^n)$ such that $|D^{s}u|(U\setminus J_{u})=0$. Moreover, we denote by~$SBV_{loc}(U;\R^{n})$ the space of functions belonging to $SBV(V;\R^{n})$ for every $V\Subset U$. For $p\in [1,+\infty)$, $SBV^{p}(U;\R^{n})$ stands for the set of functions $u\in SBV(U;\R^{n})$ with approximate gradient $\nabla{u}\in L^{p}(U;\mathbb{M}^{n})$ and $\HH^{n-1}(J_{u})<+\infty$. 

We say that $u \in GSBV( U; \R^{n})$ if $\varphi(u)\in SBV_{loc}(U; \R^{n})$ for every $\varphi\in C^{1}(\R^{n};\R^{n})$ whose gradient has compact support. Also for~$u\in GSBV(U;\R^{n})$ the approximate gradient~$\nabla{u}$ exists $\Ll^{n}$-a.e.~in~$U$ and the jump set~$J_{u}$ is countably $(\HH^{n-1}, n-1)$-rectifiable, with approximate unit normal vector~$\nu_{u}$. For $p\in [1,+\infty)$, we define $GSBV^{p}(U;\R^{n})$ as the set of functions $u\in GSBV(U;\R^{n})$ such that $\nabla{u}\in L^{p}(U;\mathbb{M}^{n})$ and $\HH^{n-1}(J_{u})<+\infty$. We refer to ~\cite[Sections~3.6,~3.9, and~4.5]{amb1} for more details on the above spaces.

In a similar fashion, the space $BD(U)$ of functions of bounded deformation is defined as the set of functions~$u \in L^{1}(U;\R^{n})$ whose distributional symmetric gradient~$Eu$ is a bounded Radon measure on~$U$ with values in~$\mathbb{M}^{n}_{s}$. In particular, we can split~$Eu$ as $Eu = E^{a}u + E^{s}u$, where~$E^{a}u$ is absolutely continuous and~$E^{s}u$ is singular w.r.t.~$\Ll^{n}$. Furthermore, the density~$e(u) \in L^{1}(U; \mathbb{M}^{n}_{s})$ of~$E^{a}u$ is the approximate symmetric gradient of~$u$, meaning that for a.e.~$x \in U$ it holds
\begin{equation}\label{e:approx_grad}
\aplim_{y \to x} \, \frac{\big( u(y) - u(x) - e(u) (x) (y - x) \big) \cdot (y-x) }{|x - y|^{2}} = 0 \,.
\end{equation}

The space~$SBD(U)$ of special functions of bounded deformation is the set of $u \in BD(U)$ such that $| E^{s}u| (U \setminus J_u)=0$. For $p \in (1, +\infty)$, we further denote by~$SBD^{p}(U)$ the space of functions~$u \in SBD (U)$ such that $\HH^{n-1}(J_u)<+\infty$ and~$e(u) \in L^{p}(U; \mathbb{M}^{n}_{s})$.

We now give the definition of $GSBD(U)$, the space of generalized special functions of bounded deformation~\cite{dal}. For $u \colon U \to \R^{n}$ measurable, $\xi \in \mathbb{S}^{n-1}$, $y \in \R^{n}$, and $V \subseteq \R^{n}$, we set 
\begin{align*}
& \Pi^{\xi} := \{ z \in \R^{n}: z \cdot \xi = 0\}\,, \qquad V^{\xi}_{y} := \{t \in \R: y + t\xi \in V\}\,,\\
& \hat{u}^{\xi}_{y} := u(y + t\xi) \cdot \xi \qquad \text{for every $t \in V^{\xi}_{y}$}\,, \qquad J^{1}_{\hat{u}^{\xi}_{y}} := \{ t \in V^{\xi}_{y} : \,| [\hat{u}^{\xi}_{y}]| >1\}\,.
\end{align*}  
Then, we say that $u \in GSBD(U)$ if there exists $\lambda \in \mathcal{M}^{+}_{b}(U)$ such that for every $\xi \in \mathbb{S}^{n-1}$ one of the two equivalent conditions is satisfied \cite[Theorem~3.5]{dal}:
\begin{itemize}

\item for every $\theta \in C^{1}(\R; [-\tfrac{1}{2}; \tfrac{1}{2}])$ such that $0 \leq \theta' \leq 1$, the partial derivative $D_{\xi} (\theta (u \cdot \xi))$ belongs to~$\mathcal{M}_{b}(U)$ and $| D_{\xi} (\theta (u \cdot \xi)) | (B) \leq \lambda(B)$ for every Borel subset~$B$ of~$U$;

\item for $\HH^{n-1}$-a.e.~$y \in \Pi_{\xi}$ the function~$\hat{u}^{\xi}_{y}$ belongs to $SBV_{loc}(U^{\xi}_{y})$ and
\begin{displaymath}
\int_{\Pi^{\xi}} \big| (D \hat{u}^{\xi}_{y}) \big| \big( B^{\xi}_{y} \setminus J^{1}_{\hat{u}^{\xi}_{y}} \big) + \HH^{0} \big( B^{\xi}_{y} \cap J^{1}_{\hat{u}^{\xi}_{y}} \big) \, \di \HH^{n-1}(y) \leq \lambda(B)
\end{displaymath}
for every Borel subset~$B$ of~$U$.

\end{itemize}
For $u \in GSBD(U)$, the approximate symmetric gradient~$e(u)$ in~\eqref{e:approx_grad} exists a.e.~in~$U$ and belongs to~$L^{1}(U;\mathbb{M}_{s}^{n})$. Its components are denoted by~$e_{i,j}(u)$ for $i, j \in \{1, \ldots, n\}$. The jump set~$J_u$ is countably $(\HH^{n-1}, n-1)$-rectifiable with approximate unit normal vector~$\nu_{u}$. If~$\hat{\mu}_{u} \in \mathcal{M}^{+}_{b}(U)$ is as in~\cite[Definitions~4.8,~4.10, and~4.16]{dal} and we set
\begin{equation}\label{e:THETA}
\Theta_{u}:= \Big\{ x \in U : \, \limsup_{r \searrow 0} \, \frac{\hat{\mu}_{u} (B_{r}(x) )} { r^{n} } >0 \Big\}\,,
\end{equation}
it was proven in~\cite[Proposition~6.1 and Theorem~6.2]{dal} that~$\Theta_{u}$ is countably $(\HH^{n-1}, n-1)$-rectifiable and coincides with~$J_u$, up to a set of~$\HH^{n-1}$-measure zero. If~$U$ has a Lipschitz boundary~$\partial U$ and $v \in GSBD(U)$, there exists a function $Tr(v) \colon \partial U \to \R^{n}$ such that for $\HH^{n-1}$-a.e.~$x \in \partial U$ 
\begin{displaymath}
Tr(v) (x) = \aplim_{\substack {y \to x\\ y \in \Om}} \, v(y) \,. 
\end{displaymath}
We refer to~$Tr(v)$ as the trace of~$v$ on~$\partial U$. Finally, for $p \in (1, +\infty)$ we say that $u \in GSBD^{p}(U)$ if $e(u) \in L^{p}(U; \mathbb{M}^{n}_{s})$ and $\HH^{n-1}(J_u) <+\infty$. We further refer to~\cite{dal} for an exhaustive discussion on the fine properties of functions in~$GSBD(U)$.


\section{Setting of the problem and main results}\label{s:setting}

In this section we present the setting of the problem and the main results of the paper. We start by discussing the energy functional that we consider in the non-rescaled reference configuration. Let $\om$ be an open bounded subset of~$\R^{n-1}$ with Lipschitz boundary~$\partial\om$. As we aim at deducing a model of brittle fracture on thin films moving from the variational theory of brittle fractures in
linearly elastic materials~\cite{fra}, we endow~$\om$ with a fictitious thickness~$\rho>0$ and define~$\Om_{\rho} := \om \times (- \frac{\rho}{2}, \frac{\rho}{2})$. Therefore, the starting point of our analysis is the functional
\begin{equation}\label{e:griffith}
\F_{\rho}(u) : = \frac{1}{2} \int_{\Om_{\rho}} \C e(u) {\, \cdot\,} e(u) \, \di x + \HH^{n-1}(J_{u}) \,,
\end{equation}
where the displacement~$u \colon \Om_{\rho} \to \R^{n}$ belongs~$GSBD^{2}(\Om_{\rho})$ and~$\C$ stands for the usual linear elasticity tensor. In a fracture mechanics setting~\cite{fra, grif}, the volume integral in~\eqref{e:griffith} is the stored elastic energy, while the surface term denotes the energy dissipated by the production of a fracture set~$J_u$. We assume in~\eqref{e:griffith} that the elastic body~$\Om_{\rho}$ is isotropic and homogeneous outside the crack set~$J_u$, so that~$\C$ is characterized by constant Lam\'e coefficients~$\lambda, \mu$, namely,
\begin{displaymath}
\C \mathrm{E} := \lambda tr(\mathrm{E}) \mathrm{I} + 2 \mu \mathrm{E}\qquad \text{for every $\mathrm{E} \in \mathbb{M}^{n}_{s}$}\,,
\end{displaymath} 
where~$\mathrm{I}$ is the identity matrix and $tr(\mathrm{E})$ denotes the trace of the matrix~$\mathrm{E}$. We notice that we may also consider, for instance, a continuous dependence of~$\lambda$ and~$\mu$ on the space variable, without changing the results of this work and with minor modifications of the proofs. As usual, we assume~$\C$ to be positive definite on the space~$\M^{n}_{s}$, which is equivalent to require~$\mu >0$ and $2 \mu  + n \lambda >0$.

As it is customary in dimension reduction, we rescale the energy functional~$\F_{\rho}$ to the fixed domain~$\Om_{1} = \omega \times \big( -\frac12, \frac12 \big)$, the so called rescaled configuration. In order to precisely describe this step, let us simplify by assuming for a moment that the displacement~$u$ actually belongs to~$C^{1}(\Om_{\rho}\setminus M_{\rho}; \R^{n})$ for a smooth hypersurface~$M_{\rho} \subseteq \Om_{\rho}$, and that~$u$ has a jump discontinuity along~$M_{\rho}$, so that $J_{u} = M_{\rho}$. 
For every $x= (x', x_{n}) \in \Om_1$ we now set
\begin{equation}\label{e:rescaled_u}
\psi_{\rho}(x) := (x', \rho x_{n})\,, \qquad v (x) := \big( u_1 ( \psi_{\rho}(x) ) , \ldots , \rho u_n (\psi_{\rho}(x) ) \big)
\end{equation}
and notice that for $x \in \Om_1$ and $\alpha = 1 , \ldots, n-1$ it holds
\begin{eqnarray}
&& \displaystyle \vphantom{\frac12}  e_{\alpha, \beta} (u) (\psi_{\rho}(x)) = e_{\alpha, \beta} (v) (x) = : e^{\rho}_{\alpha, \beta}(v) (x) \,, \label{e:rescaled_strain1} \\
&& \displaystyle   \, e_{\alpha, n} (u) ( \psi_{\rho}(x)) = \frac{1}{\rho} e_{\alpha, n} (v) (x) = : e^{\rho}_{\alpha, n}(v) (x)  \,, \label{e:rescaled_strain2}\\
&&\displaystyle     e_{n,n} (u) (\psi_{\rho}(x)) = \frac{1}{\rho^{2}} e_{n,n}(v) (x) = : e^{\rho}_{n,n}(v) (x) \label{e:rescaled_strain3}\,.
\end{eqnarray}
We further define
\begin{equation}\label{e:phirho}
\phi_{\rho} (\nu) := \bigg| \bigg( \nu_{1}, \ldots, \frac{1}{\rho} \, \nu_{n} \bigg) \bigg| \qquad \text{for every $\rho>0$ and every $\nu \in \R^{n}$}\,.
\end{equation}
By a change of coordinate and using the notation~\eqref{e:rescaled_strain1}--\eqref{e:phirho}, we rewrite~\eqref{e:griffith} computed for $u \in C^{1}(\Om_{\rho}\setminus M_{\rho}; \R^{n})$ as
\begin{equation}\label{e:G_rho}
\G_{\rho} (v) := \frac{\rho}{2} \int_{\Om_1} \C e^{\rho} (v) {\, \cdot\,} e^{\rho}(v) \, \di x + \rho \int_{J_{v}} \phi_{\rho}(\nu_{v}) \, \di \HH^{n-1} \,,
\end{equation}
where $v \in C^{1}(\Om_1 \setminus J_{v}; \R^{n})$ is as in~\eqref{e:rescaled_u} and $J_{v} = \psi^{-1}_{\rho}(M_{\rho})$.
Finally, we extend the functional~\eqref{e:G_rho} to $v \in GSBD^{2}(\Om_1)$ and define
\begin{equation}\label{e:griffith3}
\E_{\rho}(v) := \left\{ \begin{array}{ll}
\displaystyle  \frac{1}{\rho} \, \G_{\rho}(v) & \text{for $v \in GSBD^{2}(\Om_1)$}\,,\\[2mm]
\displaystyle \vphantom{\int}+\infty & \text{otherwise in $L^{0}(\Om_1)$}\,.
\end{array}\right.
\end{equation}

We now study the limit of~$\E_{\rho}$ as the thickness parameter~$\rho$ tends to~$0$. Before giving the exact expression of the limit functional, however, we investigate the closedness of a converging sequence~$u_{\rho} \in GSBD^{2}(\Om_1)$ equi-bounded in energy.

\begin{proposition}
\label{p:compactness}
Let $u_{\rho} \in GSBD^{2}(\Om_1)$ and $u \colon \Om_1 \to \R^{n}$ measurable be such that
\begin{equation}\label{e:bound}
\sup_{\rho>0} \, \E_{\rho}(u_{\rho}) < +\infty
\end{equation}
and $u_{\rho} \to u$ in measure as $\rho\to0$. Then, $u \in GSBD^{2}(\Om_1)$, $e(u_{\rho}) \rightharpoonup e(u)$ weakly in~$L^{2}(\Om_1; \M^{n}_{s})$, $e_{i,n}(u)=0$ and $e_{i, n}(u_{\rho}) \to 0$ in $L^{2}(\Om_{1})$  for $i = 1, \ldots, n$, and~$(\nu_{u})_{n} = 0$ $\HH^{n-1}$-a.e.~on~$J_u$.
\end{proposition}

\begin{proof}
From~\eqref{e:bound} we clearly deduce that~$e(u_{\rho})$ is bounded in~$L^{2}(\Om_1; \M^{n}_{s})$ and admits, up to a subsequence, a weak limit~$f \in L^{2}(\Om_1; \M^{n}_{s})$. Since $u_{\rho} \to u$ in measure in~$\Om_1$, from~\eqref{e:bound} we also deduce that $u \in GSBD^{2}(\Om_1)$ with $e(u) = f$. Arguing by slicing as in, e.g.,~\cite[Theorem~11.3]{dal}, we deduce that $u \in GSBD^{2}(\Om_{1})$ and that $e(u_{\rho}) \rightharpoonup e(u)$ weakly in~$L^{2}(\Om_{1}; \mathbb{M}^{n}_{s})$.

By definition of~$\E_{\rho}$ we have that
\begin{displaymath}
\| e_{\alpha,  n} (u_{\rho}) \|_{2}^{2} = \rho^{2} \| e^{\rho}_{ \alpha,  n} (u_{\rho}) \|_{2}^{2} \leq \rho^{2} \E_{\rho}(u_{\rho})
\end{displaymath}
and similarly $\| e_{n,n} (u_{\rho}) \|_{2}^{2} \leq \rho^{4} \E_{\rho} (u_{\rho})$. Hence,~\eqref{e:bound} implies that $e_{i,n} (u_{\rho}) \to 0$ in $L^{2}(\Om_{1} ; \M^{n}_{s})$, from which we deduce that $e_{i,n} (u) = 0$ for $i = 1, \ldots, n$.

Finally, for every $\tilde{\rho}>0$ we have that
\begin{align*}
\frac{1}{\tilde{\rho}} \int_{J_{u}} | ( \nu_{u})_{n} | \, \di \HH^{n-1} & \leq \int_{J_{u}} \phi_{\tilde{\rho}}( \nu_{u}) \, \di \HH^{n-1} \leq \liminf_{\rho \to 0} \int_{J_{u_{\rho}}} \phi_{\tilde\rho}(\nu_{u_{\rho}}) \, \di \HH^{n-1} 
\\
&
\leq \liminf_{\rho \to 0}  \int_{J_{u_{\rho}}} \phi_{\rho}(\nu_{u_{\rho}}) \, \di \HH^{n-1}  \leq \liminf_{\rho \to 0} \E_{\rho} (u_{\rho})\,.
\end{align*}
Letting~$\tilde{\rho} \to 0$ in the previous inequality and using again~\eqref{e:bound} we infer that $(\nu_{u})_{n} = 0$ $\HH^{n-1}$-a.e.~on~$J_{u}$.
\end{proof}

In view of Proposition~\ref{p:compactness}, we expect the limit functional to be defined on the space \begin{align}\label{e:KL}
\mathcal{KL} (\Om_{1}) := \{ u \in GSBD^{2}(\Om_{1}): & \, e_{i, n} (u) = 0 \text{ in~$\Om_{1}$ for $i = 1, \ldots, n$,}
\\ 
&
\,\text{and $(\nu_{u})_{n} = 0$ on~$J_{u}$}\}\,. \nonumber
\end{align}
We further denote by~$\mathcal{KL}(U)$ the same space defined on a generic open subset~$U$ of~$\R^{n}$. 

The description of~$\mathcal{KL}(\Om_{1})$ is completed by the following theorem, which therefore collects the properties of the functions~$u \in GSBD^{2}(\Om_{1})$ obtained as limits of sequences~$u_{\rho}$ equi-bounded in energy. The proof of the theorem is given in Section~\ref{s:proofs} (see, in particular, Propositions~\ref{p:badcube}--\ref{p:KL2}).

\begin{theorem}
\label{t:KL}
Let $u \in \mathcal{KL}(\Om_{1})$.
Then, the following facts hold:
\begin{itemize}
    \item[$(i)$] $u_n$ does not depend on $x_n$ and it is approximately differentiable for $\mathcal{H}^{n-1}$-a.e.~$x' \in \omega$. Moreover, denoting by~$\nabla u_n$ its approximate gradient, we have $\nabla u_n \in GSBD^2(\omega)$;
    \item[$(ii)$]  for $\mathcal{L}^n$-a.e.~$(x',x_n) \in \Om_{1}$ we have
    \begin{equation}
        \label{formuladd}
        u_\alpha(x',x_n) = \overline{u}_{\alpha}(x') - x_n \partial_\alpha u_n(x'), \ \qquad\alpha = 1, \dotsc, n-1,
    \end{equation}
    where $\overline{u}_\alpha (x'):= \int_{- 1/2}^{1/2} u_\alpha(x',x_n) \, \di x_n$ and $\overline{u} := (\overline{u}_1, \dotsc , \overline{u}_{n-1}) \in GSBD^2(\omega)$;
     \item[$(iii)$]  $J_{u} = (J_{\overline{u}} \cup J_{u_n} \cup J_{\nabla u_n} ) \times \big( -\frac12,\frac12 \big)$.
\end{itemize}
\end{theorem}

\begin{remark}\label{r:KL}
In view of Theorem~\ref{t:KL} we have that the space~$\mathcal{KL}(\Om_{1})$ in~\eqref{e:KL} is $(n-1)$-dimensional in nature, as the out of plane component~$u_{n}$ only depends on the planar coordinates~$x'$, while the planar components~$u_{\alpha}$, $\alpha = 1, \ldots, n-1$ depends linearly on~$x_{n}$ through~\eqref{formuladd}. However, the approximate symmetric gradient~$e(u) \in \M^{n}_{s}$ can be identified with an element of~$\M^{n-1}_{s}$, since the $n$-th column and the $n$-th row are zero. The structure highlighted in Theorem~\ref{t:KL} is typical of the so called \emph{Kirchoff-Love} plate, which appears in many dimension reduction problems in elasticity. 
\end{remark}



In view of Remark~\ref{r:KL}, it is convenient to introduce the following reduced linear elasticity tensor:
\begin{equation*}
\C_{0} \e {\, \cdot\,} \e := \min_{\xi \in \R^{n}} \C \e_{\xi} {\, \cdot\,} \e_{\xi} \qquad \text{for every $\e \in \M^{n-1}_{s}$}\,,
\end{equation*}
where for every $\xi \in \R^{n}$ we have set
\begin{displaymath}
\e_{\xi} := \left( \begin{array}{ccc|c}
e_{1,1} & \cdots & e_{1, n-1} & \xi_{1} \\
\vdots & \ddots & \vdots & \vdots \\
e_{n-1, 1} & \cdots & e_{n-1, n-1} & \xi_{n-1} \\ \hline
\xi_{1} & \cdots & \xi_{n-1} & \xi_{n}
\end{array}
\right)
\end{displaymath}
By a direct computation we deduce that
\begin{displaymath}
\C_{0} \e {\, \cdot\,} \e = \frac{2 \lambda \mu}{\lambda + 2\mu}( tr(\e))^{2} + 2\mu |\e|^{2}.
\end{displaymath}

With this notation at hand, the $\Gamma$-limit of~$\E_{\rho}$ writes
\begin{equation*}
\E_{0} (u) := \left\{ \begin{array}{ll}
\displaystyle \frac{1}{2} \int_{\Om} \C_{0} e(u){\, \cdot\,} e(u) \, \di x + \HH^{n-1} ( J_{u})  & \text{if $u \in \mathcal{KL}(\Om_{1})$,}\\[2mm]
\displaystyle \vphantom{\int} + \infty & \text{otherwise in~$L^0(\Om_{1})$,}
\end{array}\right.
\end{equation*}
and we have the following convergence result.

\begin{theorem}\label{t:limit}
The sequence~$\E_{\rho}$ $\Gamma$-converges to~$\E_{0}$ w.r.t.~the topology induced by the convergence in measure.
\end{theorem}

The proof of Theorem~\ref{t:limit} is given in Section~\ref{s:proofs}.

%

\section{Proofs of Theorems~\ref{t:KL} and~\ref{t:limit}}\label{s:proofs}

We start by proving Theorem~\ref{t:KL}. Its proof is articulated in the next four propositions. The first two give an approximation result in the spirit of~\cite[Section 4, Theorem 1]{cha} and~\cite[Theorem~5]{iur}. The last two, instead, provide intermediate results for the proof of items $(i)$--$(iii)$ of Theorem~\ref{t:KL}.


We now recall the definition of good/bad hyper-cubes of an $(n-1)$-dimensional grid of~$\mathbb{R}^n$ in relation with a rectifiable set with finite $(n-1)$-dimensional Hausdorff measure.

\begin{definition}
Let $h \in \mathbb{R}^+$. The $(n-1)$-dimensional $h$-grid~$\mathcal{Q}_h^0$ centered at zero and parallel to the coordinate axis is defined as
\[
\mathcal{Q}_{h}^0 := \bigcup_{i=1}^n \,  \bigcup_{z \in h \mathbb{Z}} \{ x_i = z \}  \,.
\]
A generic $(n-1)$-dimensional $h$-grid $\mathcal{Q}_h$ parallel to the coordinate axis is obtained simply by translating of a generic vector $y \in [0,1)^n$, i.e., $\mathcal{Q}_h = \mathcal{Q}_{h}^0 +hy$. 

We say that~$Q$ is a hyper-cube of $\mathcal{Q}_h = \mathcal{Q}_h^0 +hy$ if there exists $z \in h\mathbb{Z}^n$ such that 
\[
Q = z + hy + (0,h)^n.
\]

\end{definition}

\begin{definition}\label{d:badcube}
Let $\Gamma \subset \mathbb{R}^n$ be a countably $(\mathcal{H}^{n-1},n-1)$-rectifiable set with $\mathcal{H}^{n-1}(\Gamma) <\infty$. For every $y \in \mathbb{R}^n$ we introduce the {\em directional half-neighborhood}~$J^y$ of~$\Gamma$
\[
J^y := \bigcup_{x \in \Gamma} [x, x-y] \,.
\]
Set $D := \{e_i, \ e_i \pm e_j, \  i,j =1, \dotsc,n, \ i\neq j \}$.  By using the terminology of \cite{iur}, given an $(n-1)$-dimensional $h$-grid~$\mathcal{Q}_h$, we say that a hyper-cube $Q =  z +hy +(0,h)^n$ of~$\mathcal{Q}_h$ is a \emph{bad hyper-cube} relative to~$\Gamma$ if there exist $e \in D$ and $\eta \in \{0,1 \}^n$ such that
\[
\begin{cases}
 z + hy+ h\eta \in J^{he}, \text{ with } \eta_i=0, &\text{ if } e=e_i \,, \\
 z + hy + h\eta \in J^{he}, \text{ with } \eta_i=\eta_j=0, &\text{ if } e=e_i + e_j \,,\\
 z + hy + h\eta + he_j \in J^{he}, \text{ with } \eta_i=\eta_j=0, &\text{ if } e=e_i - e_j \,.
\end{cases}
\]
Otherwise, we say that a hyper-cube of~$\mathcal{Q}_h$ is a \emph{good hyper-cube} relative to~$\Gamma$.
\end{definition}

The following proposition provides an estimate of the $\HH^{n-1}$-measure of the boundaries of the bad hyper-cubes, which will be useful in view of the approximating result of Proposition~\ref{app}.

\begin{proposition}
\label{p:badcube}
Let $\Gamma \subset \mathbb{R}^n$ be a countably $(\mathcal{H}^{n-1},n-1)$-rectifiable set with $\mathcal{H}^{n-1}(\Gamma) <\infty$, and let~$\Gamma_{j}$ be a sequence of measurable sets such that
\[
\Gamma_j \subset \Gamma \text{ for every }j \in \mathbb{N} \ \ \ \text{and} \ \ \ \sum_{j=1}^\infty \mathcal{H}^{n-1}(\Gamma_j) = L < \infty \,.
\]
Moreover, for every $j \in \mathbb{N}$, every $h>0$, and every $y \in [0,1)^{n}$, let ~$\mathcal{B}_{h,j,y}$ be the family of bad hyper-cubes of~$\mathcal{Q}_{h}^0 + hy$ relative to~$\Gamma_j$ and define
\begin{equation}\label{e:A}
A_{h,j} := \bigcup_{Q \in \mathcal{B}_{h,j,y}} Q \,.
\end{equation}
Then, for every $\delta >0$ there exists a subset $H \subset (0,1)^n$ with $\mathcal{L}^{n}((0,1)^n \setminus H) \leq \delta$ for which for every $y \in H$ there exist a sequence $h_k \searrow 0$ and a sequence $j_m \nearrow  \infty$ such that
\begin{align}   \label{e:badcube2}
&\limsup_{k \to \infty} \, \mathcal{H}^{n-1}(\partial A_{h_k,j_m}) <+\infty \qquad \text{for every $m$}\,,
\\
&
 \lim_{m \to \infty} \, \limsup_{k \to \infty} \, \mathcal{H}^{n-1}(\partial A_{h_k,j_m}) = 0 \,.    \label{e:badcube1}
\end{align}
\end{proposition}
\begin{proof}
Let~$D$ be as in Definition~\ref{d:badcube}. For every $j \in \mathbb{N}$ let us denote by~$J^e_j$ the directional half-neighborhood of~$\Gamma_j$ and define the discrete jump energy
\[
E^{y,h}(\Gamma_j) := h^n \sum_{e \in D} \sum_{z \in h\mathbb{Z}^n} \frac{\mathbbm{1}_{J_j^{he}}(z + hy)}{h|e|} \,.
\]
Notice that, by definition of bad hyper-cubes, we have
\begin{equation}
\label{e:badcube}
E^{y,h}(\Gamma_j) \geq C \, \# \mathcal{B}_{h,j,y} h^{n-1}\,, 
\end{equation}
for a positive constant~$C$ independent of~$h$ and~$j$. Moreover for every~$h$ we can give the following estimate
\[
\begin{split}
\int_{[0,1)^n} \sum_{j=1}^\infty E^{y,h}(\Gamma_j) \, \di y &= \sum_{j=1}^\infty \sum_{e \in D} \sum_{z \in h\mathbb{Z}^n} h^n\int_{[0,1)^n} \frac{\mathbbm{1}_{J^{he}_j}(z + hy)}{h|e|} \, \di y
\\
&
= \sum_{j=1}^\infty \sum_{e \in D}  \int_{\mathbb{R}^n} \frac{\mathbbm{1}_{J^{he}_j}(y)}{h|e|} \, \di y =\sum_{j=1}^\infty \sum_{e\in D} \int_{\Pi^e} \bigg(\int_{\mathbb{R}} \frac{\mathbbm{1}_{J_j^{he}}(\overline{y} + se)}{h|e|} \, \di s \bigg) \di \overline{y} 
\\
&
\leq \sum_{j=1}^\infty \sum_{e\in D}\int_{\Pi^e} \mathcal{H}^0((\Gamma_j)_{\overline{y}}^e) \, \di \overline{y} \leq c \sum_{j=1}^\infty \mathcal{H}^{n-1}(\Gamma_j) =c L \,,
\end{split}
\]
where $c = \max_{| \nu | =1} (\sum_{e \in D}|\nu \cdot e|/|e|)$. Therefore, if we set $g(y) := \liminf_{h \to 0^+} \sum_{j=1}^\infty E^{y,h}(\Gamma_j)$ and define 
\[
H := \{ y \in [0,1)^n \ | \ g(y) \leq cL/\delta \} \,,
\]
by Fatou lemma and Chebyshev inequality we get that 
\[
\mathcal{L}^n([0,1)^n \setminus H) \leq \delta \,.
\]
Moreover, if $y \in H$, we have, up to passing to a subsequence depending on~$y$, that
\begin{equation}
\label{e:fatou-2}
g(y) = \lim_{h\to 0^{+}} \sum_{j=1}^{\infty} E^{y, h}(\Gamma_{j}) \leq \frac{cL}{\delta}\,.
\end{equation}
Again by Fatou lemma we have, along the same subsequence, that
\begin{equation}
\label{e:fatou}
\sum_{j=1}^\infty \liminf_{h \to 0^+} \, E^{y,h}(\Gamma_j) \leq g(y) \leq \frac{cL}{\delta} \,.
\end{equation}
 Therefore, for every $\epsilon_1>0$ there exists $j_1 \in \mathbb{N}$ such that
 \[
 \liminf_{h \to 0^+} \, E^{y,h}(\Gamma_{j_1}) \leq \epsilon_1 \,.
 \] 
 In particular, we can find a subsequence $h^1_k \searrow 0$ such that 
\[
\lim_{k \to \infty} \, E^{y,h^1_k}(\Gamma_{j_1}) = \liminf_{h \to 0^+} \,  E^{y,h}(\Gamma_{j_1}) \leq \epsilon_1 \,.
\]
Since the bounds~\eqref{e:fatou-2}--\eqref{e:fatou} are still valid along the subsequence~$(h^1_k)_k$, given~$\epsilon_2 >0$ we can find a sufficiently large $j_2 \in \mathbb{N}$ for which
\[
\liminf_{k \to \infty} \, E^{y,h^1_k}(\Gamma_{j_2}) \leq \epsilon_2 \,.
\]
As before, we can find a subsequence $(h^2_k)_k \subset (h^1_k)_k$ such that $h^2_k \searrow 0$ and
\[
\lim_{k \to \infty} \, E^{y,h^2_k}(\Gamma_{j_2}) =  \liminf_{k \to \infty} \, E^{y,h^1_k}(\Gamma_{j_2}) \leq \epsilon_2 \,.
\]
By induction, given a sequence $\epsilon_m \searrow 0$, we can construct a sequence $j_m \nearrow \infty$ and, for every $m \in \mathbb{N}$, the subsequences $(h_k^m)_k \subset (h_k^{m-1})_k$ satisfying 
\[
\lim_{k \to \infty} \, E^{y,h^m_k}(\Gamma_{j_m}) = \liminf_{k \to \infty} \, E^{y,h^{m-1}_k}(\Gamma_{j_m}) \leq \epsilon_m \,.
\]
Setting $h_k := h_k^k$ for every~$k$, we infer that $h_k \searrow 0$ and 
\[
\lim_{k \to \infty} \, E^{y,h_k}(\Gamma_{j_m}) \leq \epsilon_m \qquad  \text{ for every }m \,.
\]

Finally, by~\eqref{e:badcube} and by definition~\eqref{e:A} of~$A_{h_{k}, j_{m}}$ we estimate, for a suitable constants~$c_{1}(n), c_{2}(n) >0$ depending only on the dimension~$n$, 
\begin{displaymath}
\begin{split}
\limsup_{k \to \infty} \, \mathcal{H}^{n-1}(\partial A_{h_k,j_m}) & \leq c_1(n) \,  \limsup_{k \to \infty}\, \# \mathcal{B}_{h_k,j_m,y}\, h_k^{n-1}  
\\
&
\leq c_{2}(n) \limsup_{k \to \infty} E^{y,h_k}(\Gamma_{j_m}) \leq c_{2}(n) \epsilon_m \,,
\end{split}
\end{displaymath}
so that~\eqref{e:badcube2} holds. By letting $m \to \infty$ in the previous inequality we infer~\eqref{e:badcube1}.
\end{proof}

We now provide an approximation result for a function~$v \in GSBD(\Om_{1})$ in terms of more regular functions~$v_{k}$ whose jump~$J_{v_{k}}$ is contained in an $(n-1)$-dimensional $h$-grid~$\mathcal{Q}_h$ and that are~$W^{1, \infty}$ out of~$\overline{J_{v_{k}}}$, as in~\cite[Theorem~5]{iur}. The main difference is that here, in order to later prove Theorem~\ref{t:KL}, we have to carefully estimate the measure~$\HH^{n-1}(\pi_{n}(\overline{J_{v_{k}}}))$, where~$\pi_n$ denotes the orthogonal projection of~$\mathbb{R}^n$ onto~$e_n^\bot$. We further point out that our approximation is local in space, i.e., for $\Omega' \Subset \Om_{1}$, and that we do not need to approximate~$v$ in energy, as done in~\cite{cris}. For these reasons, a construction similar to those in~\cite[Theorem~5]{iur} can be performed in our setting without the additional assumptions~$v \in L^{2}(\Om_1;\R^{n})$ and~$e(v) \in L^{2}(\Om_1;\M^{n}_s)$, which were instead crucial in~\cite{cha, iur} to guarantee the convergence in energy and to construct a recovery sequence. 


\begin{proposition}
\label{app}
Let $U \subset \mathbb{R}^n$ be open, $v \in GSBD(U)$ with $\mathcal{H}^{n-1}(J_v) < \infty$, and  $V \Subset U$ a Lipschitz regular domain with $\mathcal{H}^{n-1}(\partial V \cap J_v)=0$. Then, there exists $(v_k)_{k=1}^\infty \subset GSBD(V) \cap W^{1,\infty}(V \setminus \overline{J_{v_k}}; \mathbb{R}^n)$ such that
\begin{enumerate}[(i)]
\item
$v_k \to v$ in measure in~$V  $ as $k \to \infty$;
\item
$e(v_k) \rightharpoonup e(v)$ weakly in $L^1(V; \M^{n}_{s})$ as $k \to \infty$; 
\item for every $\xi \in \mathbb{S}^{n-1}$
\begin{displaymath}
\lim_{k \to \infty} \mathcal{H}^{n-1}(\pi_\xi (\overline{J_{v_k}}) \setminus \pi_\xi (J_v \cap V) ) = 0\,;
\end{displaymath}

\item
$Tr(v_k) \to Tr(v)$ in $\mathcal{H}^{n-1}$-measure on $\partial V$ as $k \to \infty$;

\item If $v \cdot e_{j}$ is independent of~$x_{i}$, then for $\HH^{n-1}$-a.e.~$x' \notin \pi_{e_{i}} (\overline{J_{v}})$ the function $t \mapsto v_{k}(x' + t e_{i})\cdot e_{j}$ is constant.
\end{enumerate}

\end{proposition}

\begin{proof}
First we prove that there exists a sequence $(v_h)_{h >0} \subset GSBD(V) \cap W^{1,\infty}(V \setminus \overline{J_{v_h}}; \mathbb{R}^n)$ satisfying~$(i)$, $(ii)$, $(iv)$,  and $(v)$ as $h \to 0^+$, plus the fact $J_{v_h} \subset \mathcal{Q}_h^0 + hy$ for some $y \in [0,1)^n$. In order to prove this, we proceed similary to \cite[Theorem 5]{iur}: consider for a.e. $y \in [0,1)^n$ the $(n-1)$-dimensional $h$-grid $\mathcal{Q}_h^0 +hy$ and consider the discretized function of~$v$  
\begin{displaymath}
    v_h^y(\xi) := v(\xi +hy), \ \ \xi \in h\mathbb{Z}^{n} \cap (U -hy) \,,
\end{displaymath}
and define the continuous interpolation of $v_h^y$ 
\begin{displaymath}
w_h^y(x) := \sum_{\xi \in h\mathbb{Z}^n \cap U} v_h^y(\xi) \Delta \bigg(\frac{x - (\xi +hy)}{h} \bigg) \qquad \text{for $x \in V$},
\end{displaymath}
where
\[
\Delta(x) := \prod_{i=1}^n(1-|x_i|)^+.
\]
Let us fix $V \Subset V' \Subset U$ and let us define the \emph{piecewise constant strain} in the direction $e \in D$ as
\begin{displaymath}
    E^{y,h}_e(x) := \sum_{\xi \in V'}  \frac{[(v_h^y(\xi +he) - v_h^y(\xi))\cdot e]}{h} c^y_{e,h}(\xi) \mathbbm{1}_{\xi +hy + [0,h)^n}(x), \ \ x \in V\,, 
\end{displaymath}
where $c^y_{e,h}(\xi) := 1-\mathbbm{1}_{J^{eh}}(\xi +hy)$. Notice that, since $V \Subset V' \Subset U$, then~$E^{y,h}_e$ is well defined in~$V$ for every sufficiently small $h >0$. We claim that
\begin{equation}
    \label{e:convapp}
   \lim_{h \to 0^+}\int_{[0,1)^n}\bigg(\int_{V} |E_e^{y,h}(x) - f_e(x)| \, \di x \bigg)\di y =0 \qquad \text{for every $e \in D$}, 
\end{equation}
where $f_e := e(v)e \cdot e$. In order to simplify the next computation let us set
\[
Q_h^y(\xi) := [\xi +hy+[0,h)^n] \cap V\,, \qquad  U_h^y := U -hy\,.
\]
Now we write
\begin{align}\label{e:claim-convapp}
    &\int_{[0,1)^n}\bigg(\int_{V} |E_e^{y,h}(x) - f_e(x)| \, \di x \bigg) \di y\\ 
    &= \int_{[0,1)^n} \bigg( \!\sum_{\xi \in V'}   \int_{Q^y_h(\xi)} \! \bigg| \frac{(v (\xi+hy+he) - v(\xi +hy)) \cdot e}{h}c_{e,h}^y(\xi) - f_e(x)  \bigg|  \, \di x   \bigg) \di y \nonumber
    \\
    &= \int_{V} \bigg( \!\sum_{\xi \in V'}  \int_{[0,1)^n} \!  \bigg| \frac{(v(\xi+hy+he) - v(\xi +hy)) \cdot e}{h}c_{e,h}^y(\xi)  - f_e(x)  \bigg| \mathbbm{1}_{Q^y_h(\xi)}(x)  \, \di y   \bigg) \di x \nonumber
    \\
    &\leq \int_{V} \bigg( \!\sum_{\xi \in V' }   \mint_{\xi +[0,h)^n} \! \bigg| \frac{(v(z+he) - v(z)) \cdot e}{h}\mathbbm{1}_{U \setminus J^{eh}}(z)   - f_e(x)  \bigg| \mathbbm{1}_{[0,1)^n}\bigg(\frac{x-z}{h}\bigg)  \, \di z   \bigg)\di x\,, \nonumber
\end{align}
where in the last inequality we have performed the change of variable $z = \xi + hy$ and we have used the trivial inclusion $[0,1)^n \cap (\frac{V-\xi}{h} - y) \subset [0,1)^n$. We can continue the estimate~\eqref{e:claim-convapp} by noticing that the cubes $\xi + [0,h)^n$ are pairwise disjoints, so that 
\begin{align}\label{e:claim-convapp2}
    &\int_{[0,1)^n}\bigg(\int_{V} |E_e^{y,h}(x) - f_e(x)| \, \di x \bigg)\di y
    \\ 
    &
    \qquad \leq \int_{V} \bigg( \frac{1}{h^n}  \int_{U} \bigg| \frac{(v(z+he) - v(z)) \cdot e}{h}\mathbbm{1}_{U \setminus J^{eh}}(z)   - f_e(x)  \bigg| \mathbbm{1}_{[0,1)^n}\bigg(\frac{x-z}{h}\bigg)  \, \di z   \bigg)\di x \nonumber
    \\
    &
    \qquad \leq \int_{V} \bigg( \frac{1}{h^n}  \int_{U \setminus J^{eh}} \bigg| \frac{(v(z+he) - v(z)) \cdot e}{h}   - f_e(x)  \bigg| \mathbbm{1}_{[0,1)^n}\bigg(\frac{x-z}{h}\bigg)  \, \di z   \bigg)\di x \nonumber
     \\
    &
    \qquad \qquad+ \int_V \bigg( \frac{1}{h^n} \int_{J^{eh}} |f_e(x)| \mathbbm{1}_{[0,1)^n}\bigg(\frac{x-z}{h}\bigg)\di z\bigg)\di x\,. \nonumber
\end{align}
We treat the last two integrals in~\eqref{e:claim-convapp2} separately. For the first we have that
\[
\begin{split}
    &\int_{V} \bigg( \frac{1}{h^n}  \int_{U \setminus J^{eh}} \bigg| \frac{(v(z+he) - v(z)) \cdot e}{h}   - f_e(x)  \bigg| \mathbbm{1}_{[0,1)^n}\bigg(\frac{x-z}{h} \bigg) \, \di z\bigg)\di x 
    \\
    &
    =\int_{V} \bigg( \frac{1}{h^n}  \int_{\Pi^e } \bigg( \int_{(U \setminus J^{eh})_y^e} \bigg| \frac{(v(y+te+he) - v(y+te)) \cdot e}{h}   - f_e(x)  \bigg| \mathbbm{1}_{[0,1)^n}\bigg(\frac{x-y -te}{h} \bigg) \, \di t\bigg)\di y\bigg)\di x 
    \\ 
    &
    \leq \int_{V} \bigg( \frac{1}{h^n}  \int_{\Pi^e } \bigg( \int_{(U \setminus J^{eh})_y^e}  \bigg(\mint_0^h |D_s v(y+te+se) \cdot e    - f_e(x)|\, ds\bigg)   \mathbbm{1}_{[0,1)^n}\bigg(\frac{x-y -te}{h} \bigg) \, \di t\bigg)\di y\bigg)\di x 
    \\
    &
    = \int_{V} \bigg( \frac{1}{h^n}  \int_{\Pi^e } \bigg( \int_{(U \setminus J^{eh})_y^e}  \bigg(\mint_0^h |f_e(y+te+se)   - f_e(x)|\, \di s\bigg)   \mathbbm{1}_{[0,1)^n}\bigg(\frac{x-y -te}{h} \bigg) \, \di t\bigg)\di y\bigg)\di x, 
    \end{split}
\]
where in the last equality we have used the fact that $t \notin (U \setminus J^{eh})_y^e$ implies $\{t+s \ | \ s \in  [0,h)\} \cap (J_u)_y^e = \emptyset$. We can continue the previous estimate with
\begin{align*}
\int_{V} &  \bigg( \frac{1}{h^n}  \int_{U \setminus J^{eh}} \bigg| \frac{(v(z+he) - v(z)) \cdot e}{h}   - f_e(x)  \bigg| \mathbbm{1}_{[0,1)^n}\bigg(\frac{x-z}{h} \bigg) \, \di z\bigg) \di x
\\
&
\leq \int_{V}\bigg(\mint_0^h  \bigg( \frac{1}{h^n}  \int_{U}  |f_e(z+se)    - f_e(x)|   \mathbbm{1}_{[0,1)^n}\bigg(\frac{x-z}{h}\bigg)\di z\bigg)\di s\bigg)\di x \nonumber
\\
&
= \mint_0^h\bigg( \int_{V}\bigg(   \mint_{x - [0,h)^n}  |f_e(z+se) - f_e(x)| \,  \di z\bigg)\di x\bigg)\di s \nonumber
\\
&
=   \mint_0^h\bigg( \int_{V}\bigg(   \mint_{[0,h)^n}  |f_e(x -z + se) - f_e(x)| \,  \di z\bigg)\di x \bigg)\di s \nonumber
\\
&
=    \int_0^1\bigg(\int_{[0,1)^n}\bigg( \int_{V}|f_e(x+h( se - z )) - f_e(x)| \,  \di x\bigg)\di z\bigg)\di s \,. \nonumber
\end{align*}
The continuity property of the translations in $L^1(U)$ plus the Dominated Convergence Theorem allow us to deduce that
\begin{align}\label{e:claim-convapp3}
\lim_{h \to 0^{+}} & \int_{V}   \bigg( \frac{1}{h^n}  \int_{U \setminus J^{eh}} \bigg| \frac{(v(z+he) - v(z)) \cdot e}{h}   - f_e(x)  \bigg| \mathbbm{1}_{[0,1)^n}\bigg(\frac{x-z}{h} \bigg) \, \di z\bigg) \di x
\\
&
\leq \lim_{h \to 0^+} \int_0^1\bigg(\int_{[0,1)^n}\bigg( \int_{V}|f_e(x+h(se - z)) - f_e(x)| \,  \di x\bigg)\di z\bigg)\di s=0\,. \nonumber
\end{align}

The second term on the right-hand side of~\eqref{e:claim-convapp2} can be estimated as follows
\[
\begin{split}
\int_V \bigg( \frac{1}{h^n} \int_{J^{eh}} |f_e(x)| \mathbbm{1}_{[0,1)^n}\bigg(\frac{x-z}{h} \bigg)\di z\bigg)\di x  &\leq  \int_{J^{eh}}\bigg( \mint_{[0,h)^n} |f_e(z+x)| \, \di x  \bigg)\di z \\
&= \int_{[0,1)^n}\bigg(\int_{J^{eh} +hx}  |f_e(z)| \, \di z  \bigg)\di x \,.
\end{split}
\]
Being $\mathcal{L}^n(J^{eh})$ infinitesimal as $h \to 0^+$ (see the proof of Proposition \ref{p:badcube}), we easily deduce that
\begin{align}\label{e:claim-convapp4}
\lim_{h \to 0^{+}} \int_V &  \bigg( \frac{1}{h^n} \int_{J^{eh}} |f_e(x)| \mathbbm{1}_{[0,1)^n}\bigg(\frac{x-z}{h} \bigg)\di z\bigg)\di x 
\\
&
\leq \lim_{h \to 0^+} \int_{[0,1)^n}\bigg(\int_{J^{eh} +hx}  |f_e(z)| \, \di z  \bigg)\di x =0\,. \nonumber
\end{align}
As a consequence of~\eqref{e:claim-convapp2}--\eqref{e:claim-convapp4} we obtain the claim~\eqref{e:convapp}. Moreover, by looking at the proof of \cite[Theorem 5]{iur}, thanks to the fact that $V \Subset U$ and $\mathcal{H}^{n-1}(\partial V \cap J_v)=0$, we deduce that
\begin{align}
\label{e:app3}
    & \lim_{h \to 0^+}\int_{[0,1)^n} \bigg(\int_V|w^y_h(x) - v(x)| \wedge 1 \, \di x \bigg)\di y=0\,, \\
    \label{e:app4}
    &\lim_{h \to 0^+}\int_{[0,1)^n} \bigg(\int_{\partial V} | Tr( w^y_h) (x) - Tr( v) (x) | \wedge 1\, \di\mathcal{H}^{n-1}(x) \bigg) \di y =0\,, \\
    \label{e:app5}
    &\lim_{h \to 0^+}\int_{[0,1)^n} E^{y,h}_2((\partial V)_{nh}) \, \di y =0\,,
\end{align}
where $(\partial V)_{nh} := \{ x \in \R^{n} : \, d (x, \partial V) <nh\}$ and $ E^{y,h}_2 ( (\partial V)_{nh})$ is defined as in~\cite[formula (32)]{iur} as
\begin{displaymath}
E^{y, h}_{2} ((\partial V)_{nh}) := h^{n} \sum_{ e \in D} \sum_{\substack{\xi \in (\partial V)_{nh} - hy \\ \xi \in (\partial V)_{nh} - hy - he}} \frac{\mathbbm{1}_{J^{he} ( \xi + hy)}}{h |e|}\,.
\end{displaymath}

 We recall that since $J_v$ is countably $(\mathcal{H}^{n-1},n-1)$-rectifiable and has finite measure, arguing similarly to~\cite[Lemma 3.2.18]{fed} we find a sequence~$K_j$ of compact subsets of~$\mathbb{R}^{n-1}$ with associated Lipschitz maps $\psi_j \colon K_j \to \mathbb{R}^n$ such that $\psi_{j_1}(K_{j_1}) \cap \psi_{j_2}(K_{j_2}) = \emptyset$ for $j_1 \neq j_2$ and
\begin{equation}\label{jump0}
\mathcal{H}^{n-1}\bigg(J_v  \setminus \bigcup_{j=1}^\infty \psi_j (K_j)\bigg) = 0 \ \ \text{and} \ \ \mathcal{H}^{n-1}(\psi_j (K_j) \setminus J_v) = 0 \ \ \text{for $j \in \mathbb{N}$}\,.
\end{equation}
In addition, being $\mathcal{H}^{n-1}(\partial V \cap J_v)=0$ we may also suppose that 
\begin{equation}
\label{e:app13}
\psi_{j}(K_j) \Subset V \ \ \text{or} \ \ \psi_{j}(K_j) \Subset U \setminus \overline{V}\ \ \text{for $j \in \mathbb{N}$}\,.
\end{equation}
For every $m \in \mathbb{N}\setminus\{0\}$, let~$j_m$ be such that 
\begin{equation}
    \label{dec1}
    \sum_{j > j_m} \mathcal{H}^{n-1}(\psi_{j}(K_j)) \leq \frac{1}{m^2} \,.
\end{equation}

Let us set $\Gamma_{J_{0}} := J_{v}$ and $\Gamma_{j_{m}} := \cup_{j > j_m }\psi_j(K_j)$ for $m \geq 1$. In view of~\eqref{jump0}--\eqref{dec1} we can apply Proposition~\ref{p:badcube} from which we deduce, in combination with~\eqref{e:convapp} and~\eqref{e:app3}--\eqref{e:app5}, that there exists $y \in [0,1)^n$, a subsequence of~$(j_m)_m$, which with abuse of notation we still denote by~$(j_m)_m$, and a subsequence~$(h_k)_k$ for which we have
\begin{align}
  \label{e:app6}
  &\lim_{k \to \infty}\int_{V} |E_e^{y,h_k}(x) - f_e(x)| \, \di x =0 \ \ \text{for $e \in D$}, \\
  \label{e:app7}
  & \lim_{k \to \infty}\int_V|w^y_{h_k}(x) - v(x)| \wedge 1 \, \di x =0\,, \\
    \label{e:app8}
    &\lim_{k \to \infty}\int_{\partial V} |Tr( w^y_{h_k}) (x) - Tr(  v) (x) | \wedge 1\, \di \mathcal{H}^{n-1}(x)  =0 \,, \\
    \label{e:app9}
    &\lim_{k \to \infty} \, E^{y,h_k}_2((\partial V)_{nh_k}) =0 \,, \\
    \label{e:app10}
    &\lim_{m \to \infty} \, \limsup_{k \to \infty} \, \mathcal{H}^{n-1}(\partial A_{h_k,j_m}) = 0 \,,\\
    \label{e:app102}
    & \limsup_{k \to \infty} \, \HH^{n-1}( \partial A_{h_{k}, j_{m}} ) <+\infty \qquad \text{for every $m$}\,,
\end{align}
where~$A_{h_k,j_m}$ is the union of bad hyper-cubes of $\mathcal{Q}_{h_k}^0+h_ky$ relative to~$\Gamma_{j_m}$. We further notice that, following the proof of Proposition~\ref{p:badcube}, we may assume that the first term of the subsequence~$\Gamma_{j_{0}} = J_{v}$. Since~$y$ is fixed, in what follows we omit the dependence on~$y$.  

Now we proceed with the construction of $(v_k)_{k=1}^\infty$. Arguing similarly to \cite[Theorem~5]{iur} we define the function~$v_k$ equals to~$0$ on each bad hyper-cube of $\mathcal{Q}^0_{h_k}$ relative to~$J_v$ and $v_k := w_{h_k}$ otherwise in~$V$. In this way \eqref{e:app7}--\eqref{e:app9} imply~$(i)$ and~$(iv)$ by arguing in a very same way as in \cite[Theorem~5]{iur}, while~$(v)$ comes by construction. To prove~$(ii)$ we first notice that~\eqref{e:app6} implies in particular that
\[
E^{h_k}_e \rightharpoonup f_e, \qquad \text{weakly in }L^1(V) \text{ as } k \to \infty.
\]
By using Dunford-Pettis Theorem, we deduce the existence of a positive and increasing map $\varphi \colon \mathbb{R}^+ \to \mathbb{R}^+$ with $\lim_{t \to +\infty} \varphi(t)/t=+\infty$, for which  
\[
\sup_{k \in \mathbb{N}} \int_{V} \varphi(|E^{h_k}_e|) \, \di x < + \infty \,.
\]
On the other hand it is possible to verify that for a.e. $x$ belonging to a good hyper-cube of~$\mathcal{Q}^0_{h_k}$ relative to~$J_v$ the continuous interpolation~$w_{h_k}$ satisfies
\begin{equation}
\label{e:app12}
|e(w_{h_k})(x)e\cdot e| \leq C |E^{h_k}_e(x)| \,,  
\end{equation}
for a dimensional constant $C >0$. For instance, if $e = e_{1}$ we have that
\begin{align*}
e& (w_{h_k})(x)e_{1} \cdot e_{1}  = \sum_{\xi \in h_{k}\mathbb{Z}^{n} \cap U} e \bigg(v_{h_{k}}^{y}(\xi) \Delta \bigg(\frac{x - (\xi +h_{k} y)}{h_{k}} \bigg) \bigg) e_{1} \cdot e_{1} 
\\
&
= - \frac{1}{h_{k}} \sum_{\substack{\xi \in h_{k}\mathbb{Z}^{n} \cap U\\ x_{1} - \xi_{1} - h_{k} y_{1} \in (0,h)}} \big( v_{h_{k}}^{y}(\xi) \cdot e_{1} \big) \bigg(\prod_{j\neq 1}^n(1-|x_j - \xi_{j} - h_{k} y_{j} |)^+ \bigg) 
\\
&
\qquad + \frac{1}{h_{k}} \sum_{\substack{\xi \in h_{k}\mathbb{Z}^{n} \cap U\\ x_{1} - \xi_{1} - h_{k} y_{1} \in (-h,0)}}  \big( v_{h_{k}}^{y}(\xi) \cdot e_{1} \big) \bigg(\prod_{j\neq 1}^n(1-|x_j - \xi_{j} - h_{k} y_{j} |)^+ \bigg)
\\
&
=   \sum_{\substack{\xi \in h_{k} \mathbb{Z}^{n} \cap U\\ x_{1} -\xi_{1} - h_{k} y_{1} \in (0,h)}} \!\!\!\!\!\! \!\!\!\! \frac{\big( ( v_{h_{k}}^{y} (\xi + h_{k} e_{1}) - v_{h_{k}}^{y} (\xi) ) \cdot e_{1} \big)}{h_{k}}
 \bigg(\prod_{j\neq 1}^n(1-|x_j - \xi_{j} - h_{k} y_{j} |)^+ \bigg)  c^y_{e_{1},h_{k}}(\xi) \,,
\end{align*}
where, in the last step, we have used the fact that $x$ belongs to a good hyper-cube. Hence, we deduce~\eqref{e:app12} for $e=e_{1}$. In a similar way, we can conclude \eqref{e:app12} for every $e \in D$. 


As a consequence
\[
|e(v_k)(x)e\cdot e| \leq C |E^{h_k}_e(x)|  \qquad \text{for a.e.~$x \in V$, for $e \in D$}.
\]
For this reason, if we define the positive, increasing, and superlinear map $\psi_C \colon \mathbb{R}^+ \to \mathbb{R}^+$ as $\varphi_C(t) := \varphi(t/C)$, then we deduce
\begin{equation}
\label{e:app11}
\sup_{k \in \mathbb{N}} \int_{V} \varphi_C(|e(v_k)(x)e\cdot e|) \, \di x < + \infty \,.
\end{equation}
Since $\Gamma_{j_{0}} = J_{v}$ and, by construction, $J_{v_k} \subset \partial A_{h_k,j_0}$, by~\eqref{e:app102} we have the additional information
\begin{equation}\label{e:jump-bounded}
\sup_{k \in \mathbb{N}} \, \mathcal{H}^{n-1}(J_{v_k}) < + \infty \,.
\end{equation}
Combining~\eqref{e:app11} with~\eqref{e:jump-bounded} and~$(i)$, we can make use for example of the technique in \cite[Theorem 11.3]{dal} to deduce the validity of~$(ii)$.

To prove~$(iii)$ we fix $\xi \in \mathbb{S}^{n-1}$. To simplify the notation we denote by~$A_k$ and~$A'_k$ the union of bad hyper-cubes of $\mathcal{Q}^0_{h_k}$ relative to $J_v$ and $J_v \cap V$, respectively. By construction,~$J_{v_k}$ is contained in $\partial A_k \cap V$. We proceed as follows: first we estimate the measure of the projection of $A'_k$ onto $\xi^\bot$, then we show that the measure of the projection of $(A_k \setminus A'_k) \cap V$ onto $\xi^\bot$ is infinitesimal as $k \to \infty$, and finally we deduce~$(iii)$.

In what follows we consider only those indices $j$ for which $\psi_{j} (K_j) \Subset V$ (see \eqref{e:app13}). Let us denote by~$\mathcal{B}_{h_k,j}$ the set of bad hyper-cubes relative to~$\psi_j (K_j)$ and let~$\mathcal{B}_{h_k,j}'$ be the set of hyper-cubes for which one of their edges is contained in the set~$\{x \in V \ | \  \mathrm{dist} (x,\psi_j (K_j)) \leq h_k \}$. Then, $\mathcal{B}_{h_k,j} \subseteq \mathcal{B}'_{h_k, j}$. Now fix a direction~$\xi \in \mathbb{S}^{n-1}$. If we set $B''_{k,j} := \pi_\xi(\cup_{Q \in \mathcal{B}'_{h_k,j}} \overline{Q})$, we have that
\begin{equation}
\label{dec}
\mathcal{H}^{n-1}(B''_{k,j}\setminus \pi_\xi(\psi_{j}(K_j)) ) =  O(1/k) \,.
\end{equation}
Indeed, equality~\eqref{dec} follows from the fact that $B''_{k,j} \subset \{y \in \Pi^\xi \ | \ \text{dist}(y,\pi_\xi(\psi_j(K_j))) \leq (1 + \sqrt{n})h_k \}$ and clearly, since $\pi_\xi(\psi_j(K_j))$ is compact, it holds true
\[
\lim_{k \to \infty} \, \mathcal{H}^{n-1}\big(\{y \in \Pi^\xi \ | \ \text{dist}(y,\pi_\xi(\psi_j(K_j))) \leq (1 + \sqrt{n})h_k \} \setminus \pi_\xi(\psi_j(K_j))\big) = 0\,.
\]

In view of~\eqref{dec} given~$m$ we can find $k_m$ such that for every $j \leq j_m$ and for every $k \geq k_m$
\begin{equation}
\label{dec2}
    \mathcal{H}^{n-1}(B''_{k,j} \setminus \pi_\xi(\psi_{j}(K_j)) ) \leq  \frac{\epsilon}{j_m} \,.
\end{equation}
Let us define~$B_{k,1} := B''_{k,1} $ and, by induction,~$B_{k,j} := B''_{k,j} \setminus \cup_{l=1}^{j-1} B_{k,l}$ for every $1 < j \leq j_m$ and for every $k \geq k_m$. Notice that~\eqref{dec2} implies
\begin{align}
    \label{dec3}
\mathcal{H}^{n-1}(B_{k,j} \setminus \pi_\xi(\psi_{j}(K_j)) ) \leq  \frac{\epsilon}{j_m}  \qquad \text{for $1 \leq j \leq j_{m}$ and $k \geq k_{m}$}\,.
\end{align}

Now for every $k \geq k_m$, by construction we have that if $Q \in \mathcal{B}'_{h_{k},j}$ for some $1 \leq j \leq j_m$, then $\pi_\xi(\overline{Q}) \subset \bigcup_{j=1}^{j_m} B_{k,j}$. 
Therefore, we can use~\eqref{jump0} and~\eqref{dec3} to estimate for every $k \geq k_m$
\begin{align}\label{jump1}
& \HH^{n-1} \Bigg(  \bigg( \bigcup_{j=1}^{j_{m}}    \bigcup_{ Q \in \mathcal{B}_{h_k,j} } \pi_\xi ( \overline{Q} ) \bigg) \setminus \pi_{\xi}( J_{v} \cap V ) \Bigg) 
\\
&
\qquad \leq \HH^{n-1}   \Bigg ( \bigg( \bigcup_{j=1}^{j_{m} }  \bigcup_{ Q \in \mathcal{B}_{h_k,j} } \pi_\xi ( \overline{Q} ) \bigg)  \setminus \bigg(\bigcup_{j=1}^{\infty} \pi_{\xi}(\psi_{j}(K_{j})) \bigg)  \Bigg )  \nonumber
 \\
 &
 \qquad \leq  \HH^{n-1}  \Bigg ( \bigg(\bigcup_{j=1}^{j_{m}} \bigcup_{Q \in \mathcal{B}'_{h_k,j}} \pi_\xi(\overline{Q})  \bigg) \setminus \bigg(\bigcup_{j=1}^{\infty} \pi_{\xi}(\psi_{j}(K_{j})) \bigg)  \Bigg )  \nonumber
 \\
 &
\qquad  \leq  \HH^{n-1} \Bigg ( \bigg( \bigcup_{j=1}^{j_{m}} B_{k, j} \bigg)\setminus \bigg(\bigcup_{j=1}^{\infty} \pi_{\xi}(\psi_{j}(K_{j})) \bigg)  \Bigg )  \nonumber
 \\
 &
\qquad  \leq \sum_{j = 1}^{j_m} \mathcal{H}^{n-1} \big(B_{k, j} \setminus \pi_{\xi}( \psi_{j}(K_{j})) \big) \leq \epsilon \,. \nonumber
\end{align}

To estimate the $\mathcal{H}^{n-1}$-measure of the projection of the bad hyper-cubes relative to $J_v \cap V$ which do not belong to~$\mathcal{B}_{h_k,j}$ for some $1 \leq j \leq j_m$, we can notice that such hyper-cubes are contained in the family of bad hyper-cubes relative to~$\Gamma_{j_{m}} = \cup_{j > j_m} \psi_j(K_j)$. If we denote by~$A'_{h_{k}, j_{m}}$ the union of such bad hyper-cubes, we can use relation~\eqref{e:app10} to write
\begin{align}
\label{jump2}
\limsup_{k \to \infty} \mathcal{H}^{n-1}(\pi_\xi(\partial A'_{h_{k}, j_{m}})) &\leq \limsup_{k \to \infty} \mathcal{H}^{n-1}(\pi_\xi(\partial A_{h_{k}, j_{m}})) \\
&\leq \limsup_{k \to \infty} \mathcal{H}^{n-1}(\partial A_{h_{k}, j_{m}}) = O(1/m) \,, \nonumber
\end{align}
where in the first inequality we have used the following general fact 
\[
A' \subset A \Rightarrow \pi_\xi(\partial A') \subset \pi_\xi(\partial A),
\]
for every couple of sets $A',A \subset \mathbb{R}^n$ with $A' \subset A$ and $A$ bounded.
Now we define 
\begin{align*}
\mathcal{A}^1_k &:= \{Q \in \mathcal{Q}_{h_k}^0 : Q \text{ is a bad hyper-cubes for } J_v, \, Q \cap V \neq \emptyset,\, \overline{Q} \cap (V \setminus (\partial V)_{nh_k}) \neq \emptyset     \} \\
\mathcal{A}^2_k &:= \{Q \in \mathcal{Q}_{h_k}^0  :  Q \text{ is a bad hyper-cubes for } J_v, \, Q \cap V \neq \emptyset, \, \overline{Q} \cap (V \setminus (\partial V)_{nh_k}) = \emptyset\}.
\end{align*}
Notice that if $Q \in \mathcal{A}^1_k$ then $Q$ is a bad hyper-cube relative to $J_v$ such that $Q \subset V$. In particular, the inclusion $Q \subset V$ implies that actually $Q$ is a bad hyper-cube relative to $J_v \cap V$. Namely, the following implication holds true
\begin{equation}
    \label{e:app14}
    Q \in \mathcal{A}_k^1  \Rightarrow Q \subset A'_k.
\end{equation}
On the other hand, if $Q \in \mathcal{A}^2_k$ then $Q$ is a bad hyper-cube relative to $J_v$ such that $Q \subset (\partial V)_{nh_k}$ which means that each of its edges is contained in $(\partial V)_{nh_k}$. A similar argument to the proof of (3'') \cite[Theorem 3.5]{iur} shows that there exists a dimensional constant $c>0$ for which
\[
(\# \mathcal{A}^2_k) h^{n-1} \leq c E^{h_k}_2((\partial V)_{nh_k}).
\]
In particular we can infer
\[
\mathcal{H}^{n-1} \Big(\partial \Big(\bigcup_{Q \in \mathcal{A}^2_k} Q \Big) \Big) \leq (\# \mathcal{A}^2_k) h^{n-1} \leq c E^{h_k}_2((\partial V)_{nh_k}) \,.
\]
Condition~\eqref{e:app9} ensures that 
\begin{equation}
\label{e:app15}
\lim_{k \to \infty} \, \mathcal{H}^{n-1} \Big(\partial \Big(\bigcup_{Q \in \mathcal{A}^2_k} Q \Big) \Big) =0 \,.
\end{equation}
Every bad hyper-cube relative to~$J_v$ which has non-empty intersection with $V$ is contained in $\mathcal{A}^1_k \cup \mathcal{A}^2_k$. Therefore, if we set
\[
A^1_k := \bigcup_{Q \in \mathcal{A}_k^1} Q \ \ \text{and} \ \ A^2_k:= \bigcup_{Q \in \mathcal{A}_k^2} Q
\]
we can give the following estimate
\begin{align}\label{e:app15.1}
\mathcal{H}^{n-1} & (\pi_\xi(\partial A_k \cap V) \setminus \pi_\xi(J_v \cap V)) 
\\
&
\leq  \mathcal{H}^{n-1}(\pi_\xi(\partial A^1_k) \setminus \pi_\xi(J_v \cap V)) + \mathcal{H}^{n-1}(\pi_\xi(\partial A^2_k )) \nonumber
 \\
&\leq \mathcal{H}^{n-1}(\pi_\xi(\partial A'_k)\setminus \pi_\xi(J_v \cap V)) + \mathcal{H}^{n-1}(\pi_\xi(\partial A^2_k))\,, \nonumber
\end{align}
where for the last inequality we have used~\eqref{e:app14} to deduce that $\pi_\xi(\partial A^1_k) \subset \pi_\xi(\partial A'_k)$. We estimate separately the limsup of the last two terms of~\eqref{e:app15.1}. Concerning the first term we can use implication~\eqref{e:app14} to write 
\[
\begin{split}
\mathcal{H}^{n-1}(\pi_\xi(\partial A'_k)\setminus \pi_\xi(J_v \cap V)) \leq &\ \mathcal{H}^{n-1} (\pi_\xi ( \partial A_{h_{k}, j_{m}} ) ) \\
&+  \HH^{n-1} \Bigg(  \bigg( \bigcup_{j=1}^{j_{m} }   \bigcup_{ Q \in \mathcal{B}_{h_k,j} } \pi_\xi ( \overline{Q} ) \bigg) \setminus \pi_{\xi}( J_{v} \cap V ) \Bigg) ,
\end{split}
\]
for every~$m$, where we have used that
\[
\pi_\xi(\partial (A'_k \setminus A_{h_k,j_m})) \subset \bigcup_{j=1}^{j_{m} }   \bigcup_{ Q \in \mathcal{B}_{h_k,j} } \pi_\xi ( \overline{Q} ) \,. 
\]
Hence, we can make use of~\eqref{jump1} and~\eqref{jump2} to write
\begin{equation}
\label{e:app16}
\limsup_{k \to \infty}  \, \mathcal{H}^{n-1} (\pi_\xi (\partial A'_k) \setminus \pi_{\xi}(J_{v} \cap V) ) 
\leq O(1/m) + \epsilon  \,.
\end{equation}
The second term on the right-hand side of~\eqref{e:app15.1} can be estimated by using~\eqref{e:app15}, i.e.
\begin{equation}
    \label{e:app17}
    \limsup_{k \to \infty} \mathcal{H}^{n-1}(\pi_\xi(\partial A^2_k)) \leq \limsup_{k \to \infty} \mathcal{H}^{n-1} \Big(\partial \Big(\bigcup_{Q \in \mathcal{A}^2_k} Q\Big) \Big) =0 \,.
\end{equation}
Thanks to \eqref{e:app16}--\eqref{e:app17} and the arbitrariness of $m \in \mathbb{N}$ and $\epsilon>0$ we obtain from~\eqref{e:app15.1}
\begin{displaymath}
    \lim_{k \to \infty}\mathcal{H}^{n-1}(\pi_\xi(\partial A_k \cap V) \setminus \pi_\xi(J_v \cap V)) =0 \,.
\end{displaymath}
Finally,~$(iii)$ is proved since $\overline{J_{v_k}} \subset \partial A_k \cap V$.
\end{proof}

\begin{remark}\label{rmk:ch}
The same argument used in \cite{cha,iur} shows that whenever $v \in GSBD^2(U)$ then (ii) of Proposition \ref{app} becomes $\|e(v_k)-e(v)\|_{L^2(V)} \to 0$ as $k \to \infty$.
\end{remark}

 \begin{remark}\label{rmk:ae}
 Here we limit ourselves to observe that, in point (iii) of the previous theorem, also $\mathcal{H}^{n-1}(\pi_\xi(J_v \cap V) \setminus \pi_\xi(\overline{J_{v_k}}))$ goes to zero as $k \to \infty$ but possibly only for a.e. $\xi \in \mathbb{S}^{n-1}$.
 \end{remark}


  In the next proposition we show~$(i)$ of Theorem~\ref{t:KL} and do a first step towards the proof of formula~\eqref{formuladd}.
  
    \begin{proposition}
  \label{p:KL2}
Let $u \in \mathcal{KL}(\Omega_{1})$. 
Then,~$u_n$ does not depend on~$x_n$. Moreover,
for every $\alpha = 1, \dotsc, n-1$ there exists an $\mathcal{H}^{n-1}$-measurable function $\psi_{\alpha} \colon \omega \to \mathbb{R}$ such that
\begin{equation}
\label{formula}
 \!\!\!\  u_{\alpha}(x',x_n) = Tr ( u_\alpha) \Big(x', -\frac12 \Big) - \Big( x_n + \frac12 \Big) \psi_{\alpha}(x') \quad \text{for $\mathcal{L}^{n}$-a.e.~$(x',x_n) \in \Omega_{1}  $}\,.
\end{equation}
\end{proposition}

\begin{proof}
Combining the fact that $e_{n,n}(u) = 0$ with $(\nu_{u})_n =0$ we easily deduce that $D_n u_n =0$, so that~$u_n$ does not depend on~$x_n$.

To show formula \eqref{formula} we consider  a Lipschitz-regular open set $\omega' \Subset \omega$ such that $\mathcal{H}^{n-1}\big ( ( \partial \omega' \times (-\frac12, \frac12)) \cap J_u \big) = 0$. For $0 < \delta < \frac12$, we apply Proposition~\ref{app} to the function~$u$ on the open sets~$\omega \times (-\frac12, \frac12)$ and~$\om' \times (- \delta,  \delta)$, taking care to have chosen $\delta>0$ such that $\mathcal{H}^{n-1}(\partial (\om' \times (- \delta,  \delta)) \cap J_u)=0$ (a.e. choice of $\delta$ does the job). We denote by $(u_h)_h \subset GSBD^2(\omega' \times(- \delta,  \delta)) \cap W^{1,\infty}(\omega' \times(- \delta,  \delta) \setminus \overline{J_{u_h}}; \mathbb{R}^n)$ the approximating sequence given by Proposition~\ref{app}.

First of all notice that since~$(\nu_u)_n=0$, by property~$(iii)$ of Proposition~\ref{app} we know that $\mathcal{H}^{n-1}(\pi_n(\overline{J_{u_h}})) \to 0$ as $h \to \infty$. By passing eventually through a subsequence we may suppose $\sum_h \mathcal{H}^{n-1}(\pi_n(\overline{J_{u_h}})) < \infty$. Hence, if we define 
\begin{displaymath}
A_h := \bigcup_{k \geq h} \pi_n (\overline{J_{u_h}}) \qquad \text{and} \qquad  A := \bigcap_{h=1}^\infty A_h \,,
\end{displaymath}
then $\mathcal{H}^{n-1}(A) =0$. Moreover, from~$(i)$ and~$(iv)$ of Proposition~\ref{app} we deduce that there exists a set $I \subset (-\frac12,\frac12)$ with $\mathcal{H}^1(I)=0$ such that for $\alpha = 1, \dotsc, n-1$ the following holds true:
\begin{enumerate}
 \item $\displaystyle \lim_{h \to \infty} \int_{ \omega'}  |u_h(x',x_n) - u(x',x_n)| \wedge 1 \, \di \mathcal{H}^{n-1}(x') =0$, \ \ $x_n \in (-\frac12, \frac12) \setminus I$;
\item $\displaystyle \vphantom{\int} \lim_{h \to \infty} \int_{ \omega'}  |Tr \big( (u_h)_\alpha \big ) (x', - \delta) - Tr (u_\alpha) (x',- \delta)| \wedge 1 \, \di \mathcal{H}^{n-1}(x') =0$.
\end{enumerate}
We claim that for every $t_1,t_2 \in (-\frac12, \frac12 ) \setminus I$ we have
\begin{equation}
\label{formulad1}
 \frac{u_{\alpha}(x',t_1) - Tr (u_\alpha)  (x', - \delta)}{(t_1 + \delta )} =\frac{u_{\alpha}(x',t_2) - Tr (u_\alpha)  (x', - \delta)}{(t_2 + \delta )} \quad \mathcal{H}^{n-1}\text{-a.e. in }\omega' . 
\end{equation}


To show~\eqref{formulad1} fix $\epsilon >0$. We use conditions (1) and (2) together with Egoroff's Theorem to deduce that, up to subsequences, there exists a measurable set $E \subset \omega'$ with $\mathcal{H}^{n-1}(\omega' \setminus E) \leq \epsilon$ such that
\begin{align}
\label{e:stru1}
&\lim_{h \to \infty} \|u_h(\cdot,t_1) -u(\cdot,t_1)\|_{L^\infty(\omega'\setminus E)} =0\,, \\
\label{e:stru2}
&\lim_{h \to \infty} \|u_h(\cdot,t_2) -u(\cdot,t_2)\|_{L^\infty(\omega'\setminus E)} =0\,, \\
\label{e:stru3}
&\lim_{h \to \infty} \|Tr((u_h)_\alpha) (\cdot, - \delta ) -Tr(u_\alpha) (\cdot, - \delta )\|_{L^\infty(\omega'\setminus E)} =0\,.
\end{align}
Now let~$x' \in \omega' \setminus (A \cup E)$. Then, there exists~$h$ for which $x' \notin A_h$ it holds that $x' \in \bigcap_{k \geq h} [ \omega' \setminus \pi_n(\overline{J_{u_k}})] $. Therefore, being $\pi_n(\overline{J_{u_k}})$ closed sets, for every $k \geq h$ there exists $r >0$ (depending on $k$) for which 
\begin{equation}
\label{open}
B^{n-1}_{r}(x') \times ( - \delta, \delta) \cap \overline{J_{u_k}} = \emptyset \,,
\end{equation}
where $B^{n-1}_{r}(x') \subseteq \omega'$ denotes here the $(n-1)$-dimensional ball of radius $r$ and center~$x'$.
In particular, being~$u_n$ independent of~$x_n$, by~\eqref{open} and by~$(v)$ of Proposition~\ref{app} we have that the approximating functions~$u_k$ is such that~$(u_k)_n$ does not depend on~$x_n$ in the set $B^{n-1}_r(x') \times ( - \delta, \delta)$. Moreover, since~$u_k$ is Lipschitz continuous on $B_{r}^{n-1}(x') \times ( - \delta, \delta)$, we can apply the Fundamental Theorem of Calculus on the segment $\{x'\} \times ( - \delta,t_1)$ $(x_n < \delta)$ to deduce that, for $\alpha = 1 , \ldots, n-1$, 
\[
(u_k)_\alpha(x',t_1) - Tr \big ( (u_k)_\alpha \big)  (x', - \delta) = \int_{-\delta}^{t_1} e_{\alpha,n}(u_k)(x',t) \, \di t - \frac{(t_1 + \delta)}{2} D_{\alpha}(u_k)_n(x') \,.
\]

Hence, by using \eqref{e:stru1}, \eqref{e:stru3}, the weak convergence (ii) of Proposition \ref{app}, and the fact that $e_{\alpha,n}(u) =0$, we can take the integral on an arbitrary measurable set $B \subset \omega' \setminus (A \cup E)$ on both side of the previous inequality and let $k \to \infty$ to deduce that 
\begin{equation}
    \label{e:stru4}
    \int_{B} \frac{u_\alpha(x',t_1) - Tr \big ( u_\alpha \big)  (x', - \delta)}{t_1 + \delta} \, \di \mathcal{H}^{n-1}(x') = \lim_{k \to \infty} \int_{B}\frac{D_\alpha (u_k)_n(x')}{2} \, \di \mathcal{H}^{n-1}(x') \,.
\end{equation}
Notice that the uniform convergence \eqref{e:stru1}--\eqref{e:stru3} together with the fact that $u_k \in W^{1,\infty}([\omega' \times ( - \delta, \delta)] \setminus \overline{J_{u_h}};\mathbb{R}^n)$ guarantee that the integrand in the left hand side of~\eqref{e:stru4} belongs to $L^1(\omega' \setminus (A \cup E))$. The same argument shows that for every measurable set $B \subset \omega' \setminus (A \cup E)$ it holds true
\begin{equation}
    \label{e:stru5}
    \int_{B} \frac{u_\alpha(x',t_2) - Tr \big ( u_\alpha \big)  (x', - \delta)}{t_2+\delta} \, \di \mathcal{H}^{n-1}(x') = \lim_{k \to \infty} \int_{B}\frac{D_\alpha (u_k)_n(x')}{2} \, \di \mathcal{H}^{n-1}(x') \,.
\end{equation}
Finally, putting together \eqref{e:stru4} with \eqref{e:stru5} we deduce that 
\[
\frac{u_\alpha(x',t_1) - Tr \big ( u_\alpha \big)  (x', - \delta)}{t_1 + \delta} = \frac{u_\alpha(x',t_2) - Tr \big ( u_\alpha \big)  (x', - \delta)}{t_2 + \delta}, \qquad \mathcal{H}^{n-1}\text{-a.e. in } \omega' \setminus (A \cup E).
\]
Letting $\epsilon  \searrow 0$ in the construction of $E$, we deduce \eqref{formulad1} since $\HH^{n-1}(A) = 0$.

Now fix $t \in (-\frac12, \frac12) \setminus I$ and define the measurable set
\[
H := \bigg\{x \in \omega' \times (-\delta, \delta) \ | \ \frac{u_{\alpha}(x',x_n) - Tr (u_\alpha)  (x', - \delta)}{(x_n + \delta)} =\frac{u_{\alpha}(x',t) - Tr (u_\alpha)  (x', - \delta)}{(t + \delta)}   \bigg\}.
\]
We claim that $H$ has full measure in $\omega' \times ( - \delta, \delta)$. Indeed by using Fubini's Theorem we can write
\[
\mathcal{L}^n(H) = \int_{- \delta}^{\delta} \mathcal{H}^{n-1}(\{x' \in \omega \ | \ (x',x_n) \in H \}) \, \di x_n,
\]
which immediately implies our claim thanks to \eqref{formulad1}. By applying again Fubini's Theorem we infer that
\[
\mathcal{H}^1(\{x_n \in (-\delta, \delta) \ | \ (x',x_n) \in H  \}) = 2\delta \qquad \mathcal{H}^{n-1}\text{-a.e. }x' \in \omega'.
\]
Thus, defining
\[
\psi_{\alpha}^\delta(x') := \frac{ Tr (u_\alpha ) (x', - \delta)  - u_\alpha(x',t) }{(t + \delta)} \qquad \text{for $\mathcal{H}^{n-1}$-a.e.~$x' \in \omega'$} \,,
\]
we obtain exactly that for $\mathcal{L}^n\text{-a.e. }x=(x',x_n) \in \omega' \times (- \delta, \delta)$
\begin{equation}
    \label{formulad}
    u_\alpha(x',x_n) = Tr(u_\alpha)(x', - \delta) -(x_n + \delta) \psi_\alpha^\delta(x') \,,
\end{equation}
for every $\alpha = 1, \dotsc, n-1$. Moreover, since $Tr (u_\alpha ) (x', - \delta) \to Tr ( u_\alpha) (x',  -\frac12)$ as $\delta \to \frac{1}{2}^{+}$, defining
\[
\psi_{\alpha}(x') := \frac{ Tr ( u_\alpha)  (x', - \frac12 ) - u_\alpha(x',t) }{t + \frac12} \qquad \text{for $\mathcal{H}^{n-1}$-a.e.~$x' \in \omega'$} 
\]
and passing to the limit as $\delta \to \frac12^+$ in \eqref{formulad} (this can be done since a.e. $\delta >0$ is admissible) we obtain \eqref{formula} for $\mathcal{L}^{n}$-a.e.~$(x',x_n) \in \omega' \times (-\frac12,\frac12)$. Finally, \eqref{formula} is achieved by letting $\omega' \nearrow \omega$.
\end{proof}

\begin{proposition}
\label{vj}
Let $u \in \mathcal{KL}(\Omega_{1} )$. Then, there exists $\Gamma' \subset \omega$ such that
\[
J_u = \Gamma' \times \Big( -\frac12, \frac12 \Big) \,.
\]

Moreover, if~$\psi_{\alpha}$ are as in Proposition~\ref{p:KL2}, then the functions
 \begin{align*}
& v(x') := \Big(Tr (u_1) \Big(x',-\frac12 \Big), \dotsc, Tr ( u_{n-1} ) \Big(x',-\frac12\Big) \Big) \,, \\ 
&  \psi(x') := (\psi_1(x'), \dotsc, \psi_{n-1}(x'))
\end{align*}
belong to~$GSBD^2(\omega)$.
\end{proposition}

\begin{remark}
\label{pju3}
Notice that being the jump of~$u$ of the form $J_u = \Gamma' \times (-\frac12, \frac12)$ and being~$u_n$ independent from~$x_n$, then also~$J_{u_n}$ is of the form $\Gamma'' \times (-\frac12,\frac12)$ for some $\Gamma'' \subset \Gamma'$.
\end{remark}

\begin{proof}[Proof of Proposition~\ref{vj}]
By \cite[Theorem 4.19]{dal} we know that for $\mathcal{L}^1$-a.e. $x_n \in (-\frac12,\frac12)$ it holds true 
\[
(u_1(\cdot,x_n), \dotsc, u_{n-1}(\cdot,x_n)) \in GSBD^2(\omega) \,.
\]
In order to simplify the notation, set $w(x',x_n) := (u_1(\cdot,x_n), \dotsc, u_{n-1}(\cdot,x_n))$. Thus, by~\eqref{formula} there exist $y_n \neq z_n$ such that
\[
\frac{w(x',y_n) - w(x',z_n)}{(z_n - y_n)}= \psi(x') \in GSBD^2(\omega) \, ,
\]
which in turn, by using again formula \eqref{formula}, also implies  $v \in GSBD^2(\omega)$. This gives the second part of of the proposition.

In order to prove $J_u = \Gamma' \times (-\frac12,\frac12)$ for some~$\Gamma' \subseteq \omega$, it is enough to prove that for $\mathcal{H}^{n-1}$-a.e.~$x \in J_u$ we have 
\begin{equation}
\label{thesis}
\mathcal{H}^{n-1} \Big( \Big(\{x'\} \times \Big(- \frac12, \frac12 \Big)\Big) \cap J_u \Big) =1 \,.
\end{equation}
Suppose $x= (x',x_n) \in J_u$. Then, there are two possibilities:
\begin{enumerate}
\item
there exists $y_n \in (-\frac12, \frac12)$ such that $(x',y_n) \in J_u$;
\vspace{1mm} 
\item
$(x',t) \notin J_u$ for every $t \neq x_n$.
\end{enumerate}
In case (1), we further distinguish two subcases: either $\nu_u((x',y_n)) = \pm \nu_u((x',x_n))$ or $\nu_u((x',y_n)) \neq \pm \nu_u((x',x_n))$. In the first case, by using formula~\eqref{formula} together with the fact that~$u_n$ does not depend on~$x_n$ we have
\begin{equation}
    \label{dq}
\frac{u(x',t) - u(x',s)}{s-t} = (\psi(x'),0) \qquad \text{for }  (x',t,s) \in \omega \times \Big(-\frac12,\frac12 \Big) \times \Big(-\frac12,\frac12 \Big)\,.
\end{equation}
This implies that~$x'$ is a point of approximate continuity for~$\psi$ or a jump point for~$\psi$ with $\nu_{\psi}(x') = \pm \nu_u(x',x_n)$ (remember that $(\nu_u)_n = 0$). Suppose that~$x'$ is a jump point of~$\psi$ (in the case of a point of approximate continuity one can argue in the very same way). Then, there exist $a \neq b \in \mathbb{R}^{n}$ and $a' \neq b' \in \mathbb{R}^{n-1}$ such that
\[
u(x + ry) \to a\mathbbm{1}_{\{\nu_u(x) \cdot z >0\}}(y) + b\mathbbm{1}_{\{-\nu_u(x) \cdot z >0\}}(y) \,,
\]
locally in $\mathcal{L}^n$-measure as $r \to 0^+$, and 
\[
\psi(x' + ry') \to a'\mathbbm{1}_{\{\nu_u(x) \cdot z' >0\}}(y') + b'\mathbbm{1}_{\{-\nu_u(x) \cdot z' >0\}}(y') \,,
\]
locally in $\mathcal{H}^{n-1}$-measure as $r \to 0^+$. These two convergences imply that if we set $x_0 := (x',t)$ with $t \neq x_n$, by using
\[
\frac{u(x',t) - u(x',x_n)}{x_n - t} = (\psi(x'),0) \,,
\]
we deduce
\[
u(x_0 + ry) \to [a + (x_n - t)a']\mathbbm{1}_{\{\nu_u(x) \cdot z >0\}}(y) + [b + (x_n - t)b']\mathbbm{1}_{\{-\nu_u(x) \cdot z >0\}}(y) \,,
\]
locally in $\mathcal{H}^{n-1}$-measure as $r \to 0^+$. This means that $(x',t) \in J_u$ for every $t \in (-\frac12, \frac12) \setminus \{x_n\}$ such that
\[
a +(x_n -t)a' \neq b + (x_n-t)b' \,.
\]
Hence, \eqref{thesis} is satisfied if~$(1)$ holds and~$\nu_u(x',y_n) = \pm \nu_u(x',x_n)$.

We claim that the set of $x'$ satisfying~$(1)$ and $\nu_u( x',y_n) \neq \pm \nu_u(x',x_n)$ is $\mathcal{H}^{n-2}$-negligible. Indeed if we set $H^\pm_{x_n} := \{ y' \in \mathbb{R}^{n-1} \ | \ \pm \nu_u(x',x_n) \cdot y' > 0  \}$ and $H^\pm_{y_n} := \{ y' \in \mathbb{R}^{n-1} \ | \ \pm \nu_u(x',y_n) \cdot y' > 0  \}$, then by using again \eqref{dq} we notice that
\begin{align}
\label{conv}
\psi(x' + ry') \to  &\ \xi_1\mathbbm{1}_{H_{x_n}^+ \cap H_{y_n}^+}(y')  + \xi_2 \mathbbm{1}_{H_{x_n}^+ \cap H_{y_n}^-}(y')
\\
&
 + \xi_3 \mathbbm{1}_{H_{x_n}^- \cap H_{y_n}^+}(y') + \xi_4 \mathbbm{1}_{H_{x_n}^- \cap H_{y_n}^-}(y') \,, \nonumber
\end{align}
locally in $\mathcal{H}^{n-1}$-measure as $r \to 0^+$, where $\xi_{i}$ are vectors in $\mathbb{R}^{n-1}$ such that (as one can easily verify) the function on the right-hand side of~\eqref{conv} has a jump which is not an $(n-1)$-dimensional hyper-plane. The convergence in \eqref{conv} implies that by introducing the measure~$\hat{\mu}_{\psi}$ (see \cite[Definition 4.16]{dal}) and by exploiting the lower semi-continuity property given by \cite[Lemma 4.18]{dal}, we have $\Theta^{*(n-2)}(\hat{\mu}_{\psi},x') >0$, where $\Theta^{*(n-2)}$ denotes the upper $(n-2)$-dimensional density. This means that~$x'$ belongs to the countably $(\mathcal{H}^{n-2},n-2)$-rectifiable set~$\Theta_\psi$ (see \cite[Proposition 6.1]{dal}), which in turn, by \cite[Theorem 5.2]{dal} together with the relation among $\xi_i$, means that 
\begin{equation}\label{e:106}
\mathcal{H}^{n-2}(\{x' \in \omega \ | \ \text{(1) holds and } \pm\nu_{u}(x',x_n)\neq  \nu_{u}(x',y_n)\})=0 \,.
\end{equation}
This proves our claim. Therefore, $\mathcal{H}^{n-1}$-a.e.~$x$ satisfying case (1) also fulfills~\eqref{thesis}.

Finally, suppose~(2) holds. Such points are a subset of~$J_u$, denoted here by~$A$, satisfying $\mathcal{H}^0((A)_{x'}^{e_n}) =1$ for every $x' \in \pi_{n} (A)$. Since~$(\nu_u)_{n} = 0$, an easy application of the Area Formula implies that $\mathcal{H}^{n-1}(A)=0$, and we conclude~\eqref{thesis}.
\end{proof}

We are now in a position to conclude the proof of Theorem~\ref{t:KL}.

\begin{proof}[Proof of Theorem~\ref{t:KL}]
First we prove that $u_n$ is approximately differentiable $\mathcal{L}^{n}$-a.e. in $\Omega_{1}$. In view of~\cite[Theorem 3.1.4]{fed} it is enough to prove that the approximate partial derivatives~$\partial_i u_n$ exist $\mathcal{L}^n$-a.e.~in $\Omega_{1}$ for every $i =1,\dotsc,n$. Since we already know that~$u_n$ does not depend on~$x_n$, we need only to prove ~$\partial_\alpha u_n$ exist $\mathcal{L}^{n}$-a.e.~in $\Omega_{1}$ for every $\alpha = 1, \dotsc , n-1$.

Given~$\alpha$, we notice that since $u \in GSBD^2(\Omega_{1})$, setting $\xi := (e_n + e_\alpha)/\sqrt{2}$ we have that~$\partial_\xi (u\cdot \xi)$ and~$\partial_\alpha u_\alpha$ exist $\mathcal{L}^n$-a.e.~in $\Omega_{1}$ and by formula~\eqref{formula} also $\partial_n u_\alpha$ exists $\mathcal{L}^n$-a.e.~in $\Omega_{1}$. 

We now claim that 
\begin{equation}
\label{appgra}
\partial_\alpha u_n = 2\partial_\xi (u\cdot \xi) - \partial_\alpha u_\alpha -\partial_n u_\alpha \qquad \mathcal{L}^n\text{-a.e.~in }\Omega_{1} \,.
\end{equation}
Indeed, up to a set of $\mathcal{L}^n$-measure zero we have that for every $x \in \Omega_{1}$ the following holds true:
\begin{align}
& \aplim_{h \to 0} \, \frac{u(x + h\xi)\cdot \xi - u(x) \cdot \xi }{h} = \partial_\xi (u \cdot \xi)(x) \,,\label{(i)} \\[1mm]
& \aplim_{h \to 0} \, \frac{ u_\alpha(x + he_n) - u_\alpha(x)}{h} = \partial_n u_\alpha(x) \,, \label{(ii)} \\[1mm]
& \aplim_{h \to 0} \, \frac{ u_\alpha(x + he_\alpha) - u_\alpha(x)  }{h} = \partial_\alpha u_\alpha(x) \,, \label{(iii)} \\[1mm]
& \aplim_{h \to 0 } \, \psi (x' + h e_{\alpha}) = \psi(x')\,. \label{(iv)}
\end{align}
By a simple algebraic computation we can write
\begin{align}\label{e:100}
u_n & (x + he_\alpha) - u_n(x) 
\\
&
= u_n(x + he_\alpha) - u_n(x + he_\alpha + h e_n) + u_n(x + he_\alpha + h e_n) - u_n(x)  \nonumber
\\
&
= u_n(x + he_\alpha) -u_n(x + he_\alpha + h e_n) + \sqrt{2}u(x + h\sqrt{2} \xi) \cdot \xi - \sqrt{2}u(x) \cdot \xi \nonumber
\\
&
\qquad  - ( u_\alpha(x + h\sqrt{2}\xi) - u_\alpha(x) ) \,. \nonumber
\end{align}
By Proposition~\ref{p:KL2},~$u_{n}$ does not depend on~$x_{n}$. Thus,
\begin{equation}\label{e:101}
u_n(x + he_\alpha) -u_n(x + he_\alpha + h e_n) =0 \,.
\end{equation}
By~\eqref{(i)} we have that for $\mathcal{L}^n$-a.e.~$x \in \omega \times (0,1)$
\begin{equation}\label{e:102}
\aplim_{h \to 0}  \frac{\sqrt{2} u ( x + h \sqrt{2} \xi ) \cdot \xi - \sqrt{2} u(x) \cdot \xi  }{h} = 2  \partial_\xi (u \cdot \xi)(x)\,. 
\end{equation}
We re-write the last term on the right-hand side of~\eqref{e:100} as 
\[
u_\alpha(x + h\sqrt{2}\xi) - u_\alpha(x) = u_\alpha(x + h (e_n + e_\alpha) ) - u_\alpha(x + h e_\alpha) + u_\alpha(x + h e_\alpha) - u_\alpha(x) \,.
\]
Using formula \eqref{formula} we have that
\[
u_\alpha(x + h (e_n + e_\alpha) ) - u_\alpha(x + h e_\alpha) =- h \psi_\alpha(x'+ h e_\alpha) \,,   
\]
which implies, together with~\eqref{(iv)}, that for $\mathcal{L}^{n}$-a.e.~$x \in \Omega_{1}$
\begin{equation}\label{e:103}
\aplim_{h\to 0} \, \frac{u_\alpha(x + h (e_\alpha + e_n) ) - u_\alpha(x + h e_\alpha)}{h} = - \psi_{\alpha} (x') = \partial_{n} u_{\alpha} (x) \,,
\end{equation}
where~$\psi_{\alpha}$, $\alpha = 1, \ldots, n-1$ are the functions determined in~\eqref{formula}.
Therefore, combining~\eqref{(ii)},~\eqref{(iii)}, and~\eqref{e:103} we deduce that for $\mathcal{L}^n$-a.e.~$x \in \Omega_{1}$
\begin{equation}\label{e:104}
\aplim_{h \to 0} \, \frac{ u_\alpha ( x + h\sqrt{2} \xi ) - u_\alpha(x) }{h} = \partial_n u_\alpha(x) + \partial_\alpha u_\alpha(x) \,.
\end{equation}
Inserting~\eqref{e:101}--\eqref{e:104} in~\eqref{e:100} we obtain~\eqref{appgra}.

 Since $\alpha \in \{1, \dotsc , n-1\}$ was arbitrary, we deduce that~$u_n$ is approximately differentiable $\mathcal{L}^n$-a.e.~in $\Omega_{1}$. Furthermore, since~$u_n$ does not depend on~$x_n$,~$u_n$ is approximately differentiable $\mathcal{H}^{n-1}$-a.e.~on~$\omega$. If we denote (with abuse of notation) $\nabla u_n = (\partial_1 u_n, \dotsc, \partial_{n-1} u_n)$, then~$\nabla u_n$ is the approximate gradient of~$u_{n}$. 

In order to prove that $\nabla u_n \in GSBD^2(\omega)$, we claim that 
 \begin{equation}\label{e:105}
\nabla u_n(x') = \big(\psi_1(x'), \dotsc, \psi_{n-1}(x') \big) \qquad \text{for $\mathcal{H}^{n-1}$-a.e.~$x'\in \omega$}\,.
\end{equation}
Once we show~\eqref{e:105}, the fact that $\nabla u_n \in GSBD^2(\omega)$ will follow from Proposition~\ref{vj}. The equality~\eqref{e:105} is a consequence of the hypothesis $e_{i,n}(u) = 0$ and of~\eqref{formula}. The latter yields that $\partial_n u_\alpha = -\psi_\alpha$ $\mathcal{L}^n$-a.e.. Hence, being~$e_{\alpha, n}(u)= 0$, we infer exactly $\partial_{\alpha} u_n = \psi_\alpha$ $\mathcal{L}^{n}$-a.e., which is~\eqref{e:105}.

In order to prove \eqref{formuladd} notice that formula \eqref{formula} becomes now
\begin{equation}
\label{e: - }
 u_{\alpha}(x',x_n) = Tr ( u_\alpha)  \Big(x', - \frac12 \Big) - \Big( x_n + \frac12 \Big) \partial_{\alpha} u_{n} (x') \qquad \text{for $\mathcal{L}^{n}$-a.e.~$(x',x_n) \in \Omega_{1}$}\,.
\end{equation}
Recalling that $\Om_{1} = \omega \times ( -\frac12, \frac12) $, by integrating both sides of~\eqref{e: - } with respect to $x_n \in (-\frac12,\frac12)$ we obtain
\[
 \overline{u}_{\alpha}(x') = Tr ( u_\alpha) \Big(x', - \frac12 \Big) - \frac{1}{2}\partial_{\alpha} u_{n} (x' ) \qquad \text{for $\mathcal{L}^{n}$-a.e.~$(x',x_n) \in \Omega_{1}$}\,.
\]
Combining the last two equalities we deduce exactly~\eqref{formuladd}. The fact that $\overline{u} \in GSBD^2(\omega)$ simply follows now by~\eqref{formuladd}.

We are finally left to prove that $J_u = (J_{\overline{u}} \cup J_{u_n} \cup J_{\nabla u_n}) \times (-\frac12,\frac12)$, for which we follow the lines of~\cite[Proposition~5.2, Step 4]{bab1}. By Proposition~\ref{vj} we already know that $J_u = \Gamma' \times (-\frac12,\frac12)$ for some $\Gamma' \subset \omega$. Thus, we only need to show that $\Gamma' = J_{\overline{u}} \cup J_{u_n} \cup J_{\nabla u_n}$ up to a set of $\mathcal{H}^{n-2}$-measure zero. First, we prove $\Gamma' \subset J_{\overline{u}} \cup J_{u_n} \cup J_{\nabla u_n}$. By looking at the proof of Proposition~\ref{vj} (see equality~\eqref{dq} and the related comments), we know that for $\mathcal{H}^{n-2}$-a.e.~$x' \in \Gamma'$, either $x' \in J_{\nabla u_n}$ or $x'$ is an approximate continuity point for~$\nabla u_n$. In the first case, we clearly have $x' \in J_{\overline{u}} \cup J_{u_n} \cup J_{\nabla u_n}$. 

Let us suppose, instead, that $x'$ is an approximate continuity point of~$\nabla u_{n}$. By rewriting formula~\eqref{formuladd} in the vectorial form as
\[
u = (\overline{u}_1, \dotsc, \overline{u}_{n-1}, u_n)  - x_n   (\partial_1 u_n, \dotsc, \partial_{n-1} u_n, 0)\,,
\]
then, it is easy to see that, being~$x'$ a point of approximate continuity for~$\nabla u_n$, the fact that $x \in J_u$ forces $x' \in J_{\overline{u}} \cup J_{u_n}$. This gives the first inclusion $\Gamma' \subset J_{\overline{u}} \cup J_{u_n} \cup J_{\nabla u_n}$. 

To prove $ J_{\overline{u}} \cup J_{u_n} \cup J_{\nabla u_n} \subset \Gamma'$ we argue as follows: if $x' \in J_{u_n}$, then, by definition of~$J_{u_{n}}$,
 we have 
\begin{displaymath}
\mathcal{H}^{n-1} \Big( \Big(\{x'\} \times \Big(-\frac12,\frac12 \Big) \Big) \cap J_u \Big) = 1 \,.
\end{displaymath}
Hence, we can reduce ourselves to prove the inclusion in the case $x' \in J_{\overline{u}} \cup J_{\nabla u_n}$. Since $J_{u} = \Gamma' \times (-\frac12, \frac12)$, we can choose $\tilde{x}_n \in (-\frac12, \frac12)$ such that $v(\cdot) := u(\cdot,\tilde{x}_n) \in GSBD^2(\omega)$ and $J_v = \Gamma'$ up to a set of $\mathcal{H}^{n-2}$-measure zero in~$\omega$. Then, formula~\eqref{formuladd} says that 
\[
\Theta_{\nabla u_n} \subset \Theta_{\overline{u}} \cup \Theta_{v} \qquad  \text{and} \qquad   \Theta_{\overline{u}} \subset \Theta_{\nabla u_n} \cup \Theta_{v} \,.
\]
Since by \cite[Theorem~6.2]{dal} we know that $\mathcal{H}^{n-2}(J_f \triangle \Theta_f)=0$ for any $f \in GSBD^2(\omega)$, we deduce that, up to an $\mathcal{H}^{n-2}$-negligible set in~$\omega$,
\begin{equation}
\label{fi}
J_{\nabla u_n} \setminus J_{\overline{u}} \subset J_v = \Gamma' \qquad \text{and} \qquad J_{\overline{u}} \setminus J_{\nabla u_n} \subset J_v = \Gamma' \,.
\end{equation}
It remains to prove that
\begin{equation}
\label{si}
J_{\nabla u_n} \cap J_{\overline{u}} \subset \Gamma' \,.
\end{equation}
If $x' \in J_{\nabla u_n} \cap J_{\overline{u}}$ and $J_{\nabla u_n}, J_{\overline{u}}$ have the same tangent plane at~$x'$, by arguing as in Proposition~\ref{vj}, formula \eqref{formuladd} implies that $\mathcal{H}^{n-1} ((\{x'\} \times (-\frac12,\frac12)) \cap J_u)=1$, so that $x' \in \Gamma'$. If, instead, $x' \in J_{\nabla u_n} \cap J_{\overline{u}}$ and $J_{\nabla u_n}, J_{\overline{u}}$ have different tangent planes at~$x'$, arguing again as in the proof of~\eqref{e:106} in Proposition~\ref{vj} we obtain that the set of such points has negligible $\mathcal{H}^{n-2}$-measure. This gives~\eqref{si} and the conclusion of the Theorem.
\end{proof}

We are now in a position to prove the $\Gamma$-convergence result of Theorem~\ref{t:limit}.

\begin{proof}[Proof of Theorem~\ref{t:limit}]
We follow here the steps of~\cite[Theorem~5.1]{bab1}. Since the convergence in measure is metrizable, we can show the $\Gamma$-convergence in terms of converging sequences. As for the $\Gamma$-liminf, for every infinitesimal sequence~$\rho_{k}$, every $u\colon \Om_{1} \to \R^{n}$, and every $u_{k} \in GSBD^{2}(\Om_{1})$ such that $u_{k} \to u$ in measure and
\begin{displaymath}
\liminf_{k\to\infty}\, \E_{\rho_{k}}(u_{k}) <+\infty\,,
\end{displaymath} 
we have, in view of Proposition~\ref{p:compactness}, that $u \in \mathcal{KL}(\Om_{1})$.
Furthermore, arguing as in the proof of Proposition~\ref{p:compactness} we get that
\begin{equation}\label{e:liminf1}
\HH^{n-1}(J_u)  \leq \liminf_{k\to\infty} \int_{J_{u_{k}}} \phi_{\rho_{k}} ( \nu_{u_{\rho_{k}}}) \, \di \HH^{n-1}.
\end{equation}

For every $v \in GSBD^{2}(\Om_{1})$ let us set~$\overline{e}(v): = (e_{\alpha\beta}(v))_{\alpha, \beta=1}^{n-1}$. Then, by definition~\eqref{e:G_rho}--\eqref{e:griffith3} of~$\E_{\rho}$ we have
\begin{equation}\label{e:liminf2}
\begin{split}
\int_{\Om_1} \C_{0} e(u) {\, \cdot\,} e(u) \, \di x & \leq \liminf_{k \to \infty} \int_{\Om_1} \C_{0} \overline{e}(u_{k}) {\, \cdot \,} \overline{e}(u_{k}) \, \di x 
\\
&
\leq \liminf_{k\to\infty} \int_{\Om_1} \C e^{\rho_{k}}(u_{k}) {\, \cdot\,} e^{\rho_{k}}(u_{k}) \, \di x \,.
\end{split}
\end{equation}
Hence, combining~\eqref{e:liminf1} and~\eqref{e:liminf2} we infer that
\begin{displaymath}
\E_{0}(u) \leq \liminf_{k\to \infty}\, \E_{\rho_{k}} (u_{k})\,,
\end{displaymath}
which in turn implies that $\E_{0} \leq \Gamma$-$\liminf_{\rho \to 0} \E_{\rho}$.

We conclude with the $\Gamma$-limsup inequality. Let $u \in GSBD^{2}( \Om_1 )$. If $u \notin \mathcal{KL}(\Om_1)$, then $\E_{0}(u) = +\infty$ and there is nothing to show. If $u \in \mathcal{KL}(\Om_1)$, let us fix two sequences $h_{\rho, 1}, h_{\rho, 2} \in C^{\infty}_{c}(\om)$ such that
\begin{align}
& h_{\rho, 1} \to - \frac{\lambda}{\lambda + 2\mu} \sum_{\alpha=1}^{n-1} \partial_{\alpha} \overline{u}_{\alpha}\quad \text{ in $L^{2}(\om)$}\,, \label{e:recovery} \\[1mm]
& h_{\rho, 2} \to - \frac{\lambda}{\lambda + 2\mu} \sum_{\alpha=1}^{n-1} \partial_{\alpha} (\partial_{\alpha} u_{n}) \quad\text{in $L^{2}(\om)$}\,, \label{e:recovery2-1}\\[1mm]
& \vphantom{\int} \rho h_{\rho, 1},\, \rho h_{\rho, 2}, \, \rho \nabla h_{\rho, 1}, \, \rho \nabla h_{\rho, 2} \to 0 \qquad \text{in $L^{2}(\om)$}\,.\label{e:recovery3}
\end{align}
In particular,~\eqref{e:recovery}--\eqref{e:recovery3} and the hypothesis $u \in \mathcal{KL}(\Om_1)$ imply that
\begin{displaymath}
h_{\rho , 1}  - x_{n} h_{\rho , 2} \to -\frac{\lambda}{\lambda + 2 \mu} \sum_{\alpha=1}^{n-1} e_{\alpha, \alpha} (u) \qquad \text{in $L^{2}(\omega)$} \,.
\end{displaymath}

For every $x = ( x', x_{n}) \in \Om_1$ we define
\begin{equation}\label{e:recovery2}
u_{\rho} (x) := u (x) + \left( 0, \ldots, 0, \rho^{2} x_{n} \left[ h_{\rho, 1} (x') - \frac{x_{n}}{2}  h_{\rho, 2} (x') \right]  \right) \,.
\end{equation}
Then, $u_{\rho} \in GSBD^{2}(\Om_1)$, $J_{u_{\rho}} = J_{u}$ for every $\rho>0$, and $(\nu_{u_{\rho}})_{n} = 0$ on~$J_{u_{\rho}}$. Moreover, $u_{\rho} \to u$ in measure on~$\Om_1$, since $\rho^{2} x_{n} h_{\rho,1} \to 0 $ and $\rho^{2} x^{2}_{n}  h_{\rho, 2} \to 0$ in $L^{2}(\Om_1)$.

We now write the components of $e^{\rho}(u_{\rho})$. Since $u \in \mathcal{KL}(\Om_1)$, for every $\alpha, \beta = 1, \ldots, n-1$ we have
\begin{eqnarray*}
&& \displaystyle e_{\alpha, \beta}^{\rho} (u_{\rho}) = e_{\alpha, \beta} (u) \,, \\
&& \displaystyle  e^{\rho}_{\alpha, n} (u_{\rho}) = \frac{\rho}{2} \, x_{n} \left[ \partial_{\alpha} h_{\rho , 1}  - \frac{x_{n}}{2} \partial_{\alpha} h_{\rho , 2} \right]\,, \\
&& \displaystyle  e^{\rho}_{n,n} (u_{\rho}) = h_{\rho , 1}  - x_{n}  h_{\rho , 2}\,.
\end{eqnarray*}

Therefore, we have
\begin{align}\label{e:limsup1}
 \E_{\rho}(u_{\rho})  & =   \frac{1}{2} \int_{\Om_1} \C e^{\rho} (u_{\rho}) {\, \cdot\,} e^{\rho}(u_{\rho}) \, \di x + \int_{J_{u_{\rho}}}  \phi_{\rho} (\nu_{u_{\rho}} )  \, \di \HH^{n-1}
\\
&
= \frac{1}{2} \int_{\Om_1} \bigg[\lambda \bigg( \sum_{\alpha=1}^{n-1} e_{\alpha,\alpha}(u) \bigg)^{2} + 2 \lambda \bigg( \sum_{\alpha=1}^{n-1} e_{\alpha,\alpha}(u) \bigg) ( h_{\rho , 1} - x_{n}  h_{\rho , 2} ) \nonumber
\\
&
\qquad \vphantom{\int} + \lambda ( h_{\rho , 1} 
 - x_{n}  h_{\rho , 2})^{2} 
+ 2\mu e_{\alpha\beta}^{2} (u)  \nonumber
\\
&
\qquad \vphantom{\int} + \mu \rho^{2} x^{2}_{n} \left( \partial_{\alpha} h_{\rho, 1}  - \frac{x_{n}}{2} \partial_{\alpha} h_{\rho , 2} \right)^{2} + 2 \mu \left(h_{\rho, 1}  - x_{n}  h_{\rho,2}\right) ^{2} \bigg] \di x \nonumber
\\
&
\qquad + \HH^{n-1}( J_{u}) \,. \nonumber
\end{align}
Passing to the limit as~$\rho \to 0$ in~\eqref{e:limsup1} we get that
\begin{align*}
\lim_{\rho \to 0} \, \E_{\rho} (u_{\rho}) & = \frac{1}{2} \int_{\Om_1}\bigg[ \Big( \lambda - \frac{2\lambda^{2}}{\lambda + 2\mu} + \frac{\lambda^{3}}{( \lambda + 2\mu)^{2} } + \frac{ 2 \mu \lambda^{2}}{(\lambda + 2\mu)^{2}}\Big) (tr (e (u)))^{2} + 2\mu | e(u)| ^{2} \bigg] \, \di x 
\\
&
\qquad + \HH^{n-1}(J_{u})
= \frac{1}{2} \int_{\Om_1} \C_{0} e(u) {\, \cdot\,} e(u) \, \di x + \HH^{n-1}(J_{u}) = \E_{0}(u) \,,
\end{align*}
and the proof is thus complete.
\end{proof}

In the following corollary we show that we can naturally handle the presence of boundary conditions satisfying the properties of~\eqref{e:KL}. Despite the result follows directly from Theorem~\ref{t:limit}, it justifies the study of convergence of minima and minimizers, considered in Theorem~\ref{t:compactness2} and Corollary~\ref{c:conv_minimizer} below.

\begin{corollary}
\label{c:bc}
Let $g \in \mathcal{KL} (\R^{n} ) \cap H^{1}(\R^{n}; \R^{n})$, and let us define, for $u \in GSBD^{2} ( \Om_1 )$,
\begin{eqnarray}\label{e:Eg1}
&& \E_{\rho}^{g} (u) :=  \E_{\rho} (u) + \mathcal{H}^{n-1} \bigg( \big \{ Tr (u) \neq Tr (g) \big \} \cap \bigg( \partial\om \times \bigg(-\frac12, \frac12 \bigg) \bigg) \bigg) \,, \\[1mm]
&& \E_{0}^{g} (u) :=  \E_{0} (u) + \mathcal{H}^{n-1} \bigg( \big \{ Tr (u) \neq Tr (g) \big \} \cap \bigg( \partial\om \times \bigg( -\frac12 , \frac12 \bigg) \bigg) \bigg) \,. \label{e:Eg2}
\end{eqnarray}
Then,~$\E_{\rho}^{g}$ $\Gamma$-converges to~$\E^{g}_{0}$ w.r.t.~the topology induced by the convergence in measure in~$\Om_1$.
\end{corollary}

\begin{proof}
We consider~$\widetilde{\om} \subseteq \R^{n-1}$ smooth, bounded, and such that $\om \Subset \widetilde{\om}$, and define $\widetilde{\Om} := \widetilde{\om} \times \big ( -\frac12, \frac12 \big)$. For every $u \in GSBD^{2}(\Om_1)$, we consider the extension
\begin{equation}\label{e:extension}
\widetilde{u} := \left\{ \begin{array}{ll}
u & \text{in $\Om_1$}\,,\\
g & \text{in $\widetilde{\Om} \setminus\Om_1$}\,.
\end{array}\right.
\end{equation}
Then, we can rewrite~$\E^{g}_{\rho} (u) $ as 
\begin{eqnarray*}
\displaystyle && \E^{g}_{\rho}(u) := \frac12 \int_{\Om_{1}} \C  e^{\rho} (\tilde{u}) {\, \cdot\,} e^{\rho} (\tilde{u}) \, \di x + \int_{J_{\widetilde{u}}\cap \widetilde{\Om} } \phi_{\rho}(\nu_{\widetilde{u}}) \, \di \HH^{n-1}\,. 
\end{eqnarray*}

With this notation at hand, we can show the $\Gamma$-liminf inequality by following step by step the proof of Theorem~\ref{t:limit}. Given~$u_{\rho} \in GSBD^{2}(\Om_1)$ such that~$u_{\rho} $ converges in measure to~$u \in GSBD^{2}(\Om_1)$, we consider their extensions~$\widetilde{u}_{\rho}, \widetilde{u} \in GSBD^{2}(\widetilde{\Om})$. If
\begin{displaymath}
\sup_{\rho>0} \, \E^{g}_{\rho}(u_{\rho}) <+\infty\,,
\end{displaymath} 
we deduce that $e (u_{\rho}) \rightharpoonup e(u)$ weakly in~$L^{2}( \Om_1 ; \mathbb{M}^{n}_{s})$ and $u \in \mathcal{KL}(\Om_1)$, so that also $\widetilde{u} \in \mathcal{KL}(\widetilde{\Om})$. Furthermore, the bulk energy satisfies
\begin{displaymath}
\int_{\Om_1} \C_{0} e (u) {\, \cdot\,} e (u) \, \di x \leq \liminf_{\rho \to 0} \int_{\Om_1} \C e^{\rho}(\tilde{u}_{\rho}) {\, \cdot\,} e^{\rho} (\tilde{u}_{\rho}) \, \di x\,.
\end{displaymath}
As in~\eqref{e:liminf1} we have that
\begin{displaymath}
\HH^{n-1} (J_{\widetilde{u}} \cap \widetilde{\Om} ) \leq \liminf_{\rho \to 0} \int_{J_{\widetilde{u}_{\rho}} \cap \widetilde{\Om}} \phi_{\rho} (\nu_{\widetilde{u}_{\rho}}) \, \di \HH^{n-1}.
\end{displaymath}
Noticing that $\HH^{n-1} (J_{\tilde{u}} \cap (\widetilde{\Om} \setminus \Om_{1} ) ) = 0 $ and 
\begin{displaymath}
J_{\tilde{u}} \cap \partial \om \times \bigg( -\frac12 , \frac12 \bigg) =  \big \{ Tr (u) \neq Tr (g) \big \} \cap \bigg(\partial\om \times \bigg( -\frac12 , \frac12 \bigg) \bigg) \,,
\end{displaymath}
we deduce that $\E^{g}_{0}(u) \leq \liminf_{\rho\to 0} \E^{g}_{\rho}(u_{\rho})$.

A recovery sequence can be constructed as in~\eqref{e:recovery2}, where we modify a function~$u \in \mathcal{KL}(\Om_{1})$ within~$\Om_{1}$ by considering~$h_{\rho, 1}, h_{\rho,2} \in C_{c} (\omega)$ as in~\eqref{e:recovery}, so that~$u$ remains unchanged on~$\partial\om \times \big (-\frac12 , \frac12 \big )$.
\end{proof}

We now discuss the convergence of minimizers of the functionals~$\E^{g}_{\rho}$. To do this, we recall here the GSBD-compactness result obtained in~\cite[Theorem~1.1]{cris2} (see also~\cite{new}).

\begin{theorem}
\label{t:comfk}
Let $U \subseteq \R^{n}$ be an open bounded subset of~$\R^{n}$, let $\phi \colon \mathbb{R}^+ \to \mathbb{R}^+$ be an increasing function such that
\[
\lim_{t \to +\infty} \frac{\phi(t)}{t} = +\infty,
\]
and let $u_\rho \in GSBD^2(U)$ be such that 
\[
\sup_{\rho>0} \int_{U} \phi(|e(u_\rho) |) \, \di x + \mathcal{H}^{n-1}(J_{u_\rho}) < \infty \,.
\]
Then, there exists a subsequence, still denoted by~$u_\rho$, such that the set 
\[
A := \{x \in U \ | \ |u_\rho(x)| \to +\infty \text{ as } \rho \to 0^+  \}
\]
has finite perimeter, $u_\rho \to u$  a.e.~in~$U \setminus A$ and $e(u_{\rho}) \rightharpoonup e(u)$ weakly in~$L^{1}(U\setminus A; \M^{n}_{s})$  for some function $u \in GSBD^2(U)$ with $u=0$ in~$A$. Furthermore,
\begin{displaymath}
\HH^{n-1}(J_{u} \cup \partial^{*}A) \leq \liminf_{\rho \to 0} \, \HH^{n-1}(J_{u_{\rho}}) \,.
\end{displaymath} 
\end{theorem}

From Theorem~\ref{t:comfk} we deduce the convergence of minima and minimizers.

\begin{theorem}\label{t:compactness2}
Let $g \in \mathcal{KL} ( \R^{n} ) \cap H^{1}(\R^{n}; \R^{n} )$, and let~$\E^{g}_{\rho}$ be the sequence of functionals defined in~\eqref{e:Eg1}. Assume that $u_{\rho} \in GSBD^{2}( \Om_1)$ satisfies
\begin{equation}\label{e:hpc}
\liminf_{\rho \to 0} \, \E^{g}_{\rho} (u_{\rho}) < +\infty\,.
\end{equation}
Then, there exists a subsequence, still denoted by~$u_{\rho}$, such that the set
\begin{displaymath}
A := \{ x \in \Om_1 : \, | u_{\rho} (x) | \to +\infty \text{ as $\rho \to 0$}\}
\end{displaymath}
is a set of finite perimeter. Moreover, there exist $A' \subseteq \om$ and $u \in \mathcal{KL}(\Om_1)$ with $u=0$ in~$A$ such that
\begin{align}
& A = A' \times \bigg( - \frac12 , \frac12 \bigg)\,, \label{e:c1} \\
& u_{\rho} \to u \qquad \text{a.e.~in~$\Om_1 \setminus A$}\,, \label{e:c2} \\
& e (u_{\rho}) \rightharpoonup e(u) \qquad \text{weakly in~$L^{2}(\Om_1 \setminus A; \M^{n}_{s})$}\,, \label{e:c3} \\
& \HH^{n-1}(J_{u} \cup \partial^{*} A ) + \mathcal{H}^{n-1} \bigg(  \big\{ Tr (u) \neq Tr (g) \big\}   \cap \bigg( \partial\om \times \bigg( - \frac12 , \frac12 \bigg) \bigg) \bigg) \label{e:c4} 
\\
&
 \quad \leq \liminf_{\rho \to 0} \, \int_{J_{u_{\rho}}} \phi_{\rho}( \nu_{u_{\rho}}) \, \di \HH^{n-1} + \mathcal{H}^{n-1} \bigg( \big \{ Tr (u_{\rho}) \neq Tr (g) \big\} \cap \bigg ( \partial\om \times \bigg( -\frac12 , \frac12 \bigg) \bigg) \bigg) . \nonumber
\end{align}
\end{theorem}

\begin{proof}
Let $\widetilde{\omega}$ and $\widetilde{\Omega}$ be as in the proof of Corollary~\ref{c:bc}.Along the proof, we denote by~$\partial^{*}E$ and~$\widetilde{\partial}^{*} E$ the reduced boundary of a set~$E \subseteq\widetilde{\Om}$ in~$\Om$ and~$\widetilde{\Om}$, respectively.

The existence of the set~$A$ and of a limit function~$u \in GSBD^{2}(\Om_1)$ such that~\eqref{e:c2}--\eqref{e:c3} holds follows from~Theorem~\ref{t:comfk} applied to the sequence $\widetilde{u}_{\rho} \in GSBD^{2}(\widetilde{\Om})$ defined as in~\eqref{e:extension}. Precisely, there exists $A \subseteq \widetilde{\Om}$ and $\widetilde{u} \in GSBD^{2}(\widetilde{\Om})$ such that~\eqref{e:c2}--\eqref{e:c3} hold for $\widetilde{u}_{\rho}$ and $\widetilde{u}$ in~$\widetilde{\Om}$. Since $\widetilde{u}_{\rho} = \widetilde{u} = g$ in~$\widetilde{\Om} \setminus\Om_1$ and $g \in H^{1}(\R^{n}; \R^{n})$, we clearly have that $u:= \tilde{u} \mathbbm{1}_{\Om_{1}} \in GSBD^{2}(\Om_1)$ and~$A \subseteq \overline{\Om}_1$.

Let us denote by~$\nu_{\tilde{u} \cup \widetilde{\partial}^{*} A}$ the approximate unit normal to~$J_{\tilde{u}} \cup \widetilde{\partial}^{*} A$. By~\cite[Proposition~4.6]{kol},~$\widetilde{u}$ and~$A$ are such that
\begin{align}\label{e:c5}
\int_{J_{\widetilde u} \cup \widetilde{\partial}^{*} A} & \phi( x, \nu_{\widetilde u \cup \widetilde{\partial}^{*}A}) \, \di \HH^{n-1} \leq \liminf_{\rho \to 0} \int_{J_{\widetilde{u}_{\rho}}} \phi(x, \nu_{\widetilde{u}_{\rho}}) \, \di \HH^{n-1}
\end{align}
for every $\phi \in C(\widetilde{\Om} \times \R^{n})$ such that $\phi(x, \cdot)$ is a norm on~$\R^{n}$ for every $x \in \widetilde{\Om}$ and
\begin{displaymath}
c_{1} | \nu| \leq \phi (x, \nu) \leq c_{2} | \nu| \qquad \text{for every $x \in \widetilde{\Om}$ and every $\nu \in \R^{n}$}\,,
\end{displaymath}
for some~$0 < c_{1} \leq c_{2} < + \infty$. 

Recalling~\eqref{e:phirho}, we deduce from~\eqref{e:c5} that for every $\tilde{\rho}>0$
\begin{align}\label{e:c6}
\int_{J_{\widetilde u} \cup \widetilde{\partial}^{*} A} & \phi_{\tilde{\rho}} ( \nu_{\widetilde u \cup \widetilde{\partial}^{*} A} )\, \di \HH^{n-1} \leq \liminf_{\rho \to 0} \int_{J_{\widetilde{u}_{\rho}}} \phi_{\tilde{\rho}} ( \nu_{\widetilde{u}_{\rho}}) \, \di \HH^{n-1} 
\\
&
\leq \liminf_{\rho \to 0} \int_{J_{\widetilde{u}_{\rho}}} \phi_{\rho} (\nu_{\widetilde{u}_{\rho}}) \, \di \HH^{n-1} \nonumber
\\
&
= \liminf_{\rho \to 0} \int_{J_{u_{\rho}}} \phi_{\rho} ( \nu_{u_{\rho}}) \, \di \HH^{n-1} \nonumber
\\
&
\qquad \qquad + \mathcal{H}^{n-1} \bigg( \big \{ Tr (u_{\rho}) \neq Tr (g) \big \} \cap \bigg( \partial\om \times \bigg( - \frac12 , \frac12 \bigg) \bigg ) \bigg) <+\infty\,. \nonumber
\end{align}
Passing to the limsup in~\eqref{e:c6} as~$\tilde{\rho} \to 0$ we deduce that $(\nu_{\widetilde{\partial}^{*} A})_{n} = (\nu_{u})_{n} = 0$ $\HH^{n-1}$-a.e.~in~$J_{u} \cup \widetilde{\partial}^{*} A$. It follows that there exists~$A' \subseteq \omega$ such that~\eqref{e:c1} holds. 

As a consequence of~\eqref{e:hpc}, we infer that $e_{i,n}(u) = 0$ in~$\Om_{1}$ for every $i=1, \ldots, n$. Hence, $u \in \mathcal{KL}(\Om_{1})$. Taking into account that~$(\nu_{u})_{n} = ( \nu_{\widetilde{\partial}^{*} A})_{n} = 0$ and that  
\begin{displaymath}
J_{\tilde{u}} \cap \partial \om \times \bigg( -\frac12 , \frac12 \bigg) =  \big \{ Tr (u) \neq Tr (g) \big \} \cap \partial\om \times \bigg( -\frac12 , \frac12 \bigg) \bigg) \,,
\end{displaymath}
we infer~\eqref{e:c4} by rewriting~\eqref{e:c6}, and the proof is thus concluded.
\end{proof}

\begin{corollary}\label{c:conv_minimizer}
Under the assumptions of Theorem~\ref{t:compactness2}, let $u_{\rho} \in GSBD^{2}(\Om_1)$ be a sequence of minimizers of~$\E_{\rho}^{g}$. Then, there exist a subsequence, still denoted by~$u_{\rho}$, such that the set $A := \{ x \in \Om_1: \, | u_{\rho} (x) | \to +\infty\}$ is of finite perimeter, and a minimizer~$u \in \mathcal{KL}(\Om_1)$ of~$\E^{g}_{0}$ with~$u=0$ on~$A$ such that~\eqref{e:c2}--\eqref{e:c3} hold. Moreover, $\partial^{*} A \subseteq J_{u}$, $e(u_{\rho}) \to e(u)$ in $L^{2}(\Om_{1}; \M^{n}_{s})$, and 
\begin{align}
\label{measurelim}
 \HH^{n-1}& (J_{u} )   + \mathcal{H}^{n-1} \bigg(  \big \{ Tr (u) \neq Tr (g) \big \}   \cap \bigg( \partial\om \times \bigg(- \frac12 , \frac12 \bigg) \bigg) \bigg) 
\\
&
  \!\!= \lim_{\rho \to 0} \, \int_{J_{u_{\rho}}} \phi_{\rho}( \nu_{u_{\rho}}) \, \di \HH^{n-1} + \mathcal{H}^{n-1} \bigg(  \big \{ Tr (u_{\rho}) \neq Tr (g) \big \} \cap \bigg( \partial\om \times \bigg( - \frac12, \frac12 \bigg) \bigg) \bigg) \nonumber  \,.
\end{align}
\end{corollary}

\begin{proof}
Let $u_{\rho}$ be as in the statement of the corollary. Then, it is easy to check that~\eqref{e:hpc} is satisfied. Hence, Theorem~\ref{t:compactness2} implies that there exist~$A$ and~$u \in \mathcal{KL}(\Om_1)$ such that~\eqref{e:c1}--\eqref{e:c4} hold. The minimality of~$u$ follows from Theorem~\ref{t:limit} by the usual~$\Gamma$-convergence argument.

Finally, we notice that
\begin{align}
& \int_{\Om_1} \C_{0} e(u) {\, \cdot\,} e(u) \, \di x \leq \liminf_{\rho \to 0} \int_{\Om_1} \C e^{\rho}(u_{\rho}) {\, \cdot\,} e^{\rho} (u_{\rho}) \, \di x\,, \label{e:sm}\\[1mm]
& \vphantom{\int} \HH^{n-1}(J_{u} \cup  \partial^{*} A )  + \mathcal{H}^{n-1} \bigg(  \big \{ Tr (u) \neq Tr (g) \big \}   \cap \bigg( \partial\om \times \bigg(- \frac12 , \frac12 \bigg) \bigg) \bigg)  \label{e:sm2}
\\
&
 \quad \leq \liminf_{\rho \to 0} \, \int_{J_{u_{\rho}}} \phi_{\rho}( \nu_{u_{\rho}}) \, \di \HH^{n-1} + \mathcal{H}^{n-1} \bigg(  \big \{ Tr (u_{\rho}) \neq Tr (g) \big \} \cap \bigg( \partial\om \times \bigg( - \frac12, \frac12 \bigg) \bigg) \bigg) \nonumber  \,. 
\end{align}
Since we can construct a recovery sequence $v_{\rho} \in GSBD^{2}(\Om_1)$ for~$u$ such that $\E^{g}_{\rho}(v_{\rho}) \to \E^{g}_{0}(u)$ and~$u_{\rho}$ is a minimizer of~$\E^{g}_{\rho}$ for every $\rho$, we deduce that, along a suitable not relabeled subsequence, the inequalities \eqref{e:sm}--\eqref{e:sm2} are actually equalities. This implies that~$\partial^{*}  A \subseteq J_{u}$,~\eqref{measurelim}, and that
\begin{displaymath}
 \int_{\Om_1} \C_{0} e(u) {\, \cdot\,} e(u) \, \di x = \lim_{\rho \to 0} \, \int_{\Om_{1}} \C_{0} e(u_{\rho}) {\, \cdot\,} e(u_{\rho}) \, \di x \,.
\end{displaymath}
From the last equality and from Proposition~\ref{p:compactness} we infer that $e(u_{\rho}) \to e(u)$ in $L^{2}(\Om_{1}; \M^{n}_{s})$.
\end{proof}

\section*{Acknowledgements}

The authors would like to acknowledge the kind hospitality of the Erwin Schr\"odinger International
Institute for Mathematics and Physics (ESI), where part of this research was developed during the
workshop {\em  Modeling of Crystalline Interfaces and Thin Film Structures: A Joint Mathematics-Physics Symposium}. S.A. also acknowledges the support of the OeAD-WTZ project CZ 01/2021 and of the FWF through the project I 5149.

\bibliographystyle{siam}
\bibliography{A-T_20_bib_2.bib}

\end{document}